\documentclass{article}
\usepackage{amsmath,amsthm,amssymb,graphics,wrapfig}
\newcommand{\R}{\mathbb{R}\mkern1mu}
\newcommand{\Z}{\mathbb{Z}\mkern1mu}
\newcommand{\N}{\mathbb{N}\mkern1mu}
\newtheorem{theorem}{Theorem}
\newtheorem{lemma}{Lemma}
\newtheorem{cor}{Corollary}

\theoremstyle{remark}
\newtheorem{rem}{Remark}
\DeclareMathOperator{\Id}{Id}
\title{Fitting lines to points in the plane}
\author{Annett P\"uttmann}
\date{\today}
\begin{document}
\maketitle 
\begin{abstract}
We seek for lines of minimal distance to finitely many given points in the plane.
The distance between a line and a set of points is defined by the $L^p$-norm, $1\leq p\leq \infty$, 
of the vector of vertical or orthogonal distances from the single points to the line.  

The known properties of optimal lines are deduced by elementary considerations and 
represented using a uniform language for the different choices to define the distance from a line 
to a set of points.
\end{abstract}
\section{Introduction}
This article deals with an elementary problem.
We seek for a line that is as close as possible to a finite set of points in the plane.
The regression line is a well-known solution and it is even easy to compute.
However, it is only one out of many possible answers, since the optimal line depends on the
definition of the distance between a line and a set of points. 
\subsection{Distance between a line and a set of points}
There are several useful ways to measure the distance between a line and a finite set of points.
Usually, this is done in two steps:
First the distance $d_j$ between a single point $p_j$ and a line is defined.
Then these distances  are combined to a distance between the line and the point set  
$\{ p_1,\ldots,p_m\}$.

We consider the vertical and the orthogonal distance between a point and a line.
The vertical distance between a point $(x_j,y_j)$ and the graph of a linear function 
$y(x) =ax+b$ is defined by $d_j=|y_j-ax_j-b|$. 
To emphasize that the solution should be a graph of a (linear) function 
this distance is also called algebraic distance.
The orthogonal distance between a line $g$ and a point $p_j$ is the length 
of the line segment connecting $p_j$ and $g$ perpendicular to $g$.
To compute this distance the euclidian inner product is used. 
Therefore, this distance is also called euclidian or geometric distance.

The distance between a point set $\{ p_1,\ldots,p_m\}$ and a line is defined by the
$L^p$-norm of the vector $(d_1,\ldots,d_m)\in R^m$, $1\leq p\leq \infty$.
For $p=2$ we want to minimize the sum of squares $d_1^2+\ldots+d_m^2$ 
of the distances of the single points. This is the classical method of least squares.
The $L^2$-norm is preferred in statistics. 
For $p=1$ and $p=\infty$ we want to minimize the sum $d_1+\ldots+d_m$ respectively 
the maximum $\max\{ d_1,\ldots,d_m\}$  of the distances of the single points.
These quantities occur in optimization problems.
Then the $L^p$-norm is a natural generalization.
Furthermore, the investigation of the function $d_1^p+\ldots+d_m^p$ for all $1\leq p$
connects the extreme norms ($p=1$, $p=\infty$) to the case $p=2$.
\subsection{Goals}
The results of this article are known, partially for a long time,
(e.g.~\cite{A}, \cite{BS}, \cite{C}, \cite{K}, \cite{P}, \cite{S}, \cite{SW}) 
or can be  rather easily derived for some distances.
But the facts are spread and are often formulated according to the needs of the applications.
In this article the results are presented in a uniform language accessible to a broad audience.
The results are deduced by elementary considerations 
that require only basic linear algebra and calculus.
Therefore, the appendix contains the necessary facts on convex sets and convex functions.

Due to the homogeneous representation one can easily observe how 
the properties of the objective function $f$, the suitable methods and the set of optimal lines change
using different definitions of the distance. Here are two examples:
The algebraic distance leads to convex or even strictly convex functions.
Using the geometric distance we loose the convexity in one variable but we gain the compactness of the
domain of definition of $f$.
Differential calculus provides explicit  formulas of optimal lines for $p=2$.
For $p=1$ and  $p=\infty$ the function $f$ is not even differentiable at all points but it is piecewise linear.
So the global minimum can be determined comparing finitely many values of $f$.
 
Despite the differences of the distances investigated here 
there is one basic idea behind all the solutions. 
Properties of optimal lines are obtained observing the behavior of $f$ 
while translating and rotating a line.
If $f$ is differentiable, then partial differentiation yields a system of equations for the critical points of $f$,
else this approach gives  at least information about the location of optimal lines.

Moreover, dealing with this elementary optimization problem, 
that can be satisfactory solved by simple arguments in many cases, 
motivates further questions in different fields of mathematics.
For example, what are the effects of small changes of the given point set or the parameter $p$. 
If explicit formulas of the optimal lines are not available, then one needs 
fast algorithms to manage the input data ($p=1$ and $p=\infty$)
or to numerically approximate the solution  ($p\ne 1,2,\infty$).
It is even more challenging to fit other objects such as circles or ellipses to point sets  \cite{C}.
In these cases the basic questions of existence and uniqueness of objects with minimal distance are
already more difficult to answer.
\subsection{Symmetries}
The algebraic and the geometric distance and the associated optimal lines, behave differently
regarding coordinate changes.

The geometric distances are by definition invariant under isometries, these are
translations, rotations and reflections.
The algebraic distances are invariant under translations and 
reflections in the coordinate axes or a point. 

If the distance is invariant under an (affine) transformation $A:\R^2\to\R^2$,
then a line $g$ has minimal distance to the point set $\{p_1,\ldots,p_m\}$ if and only if
the line $A(g)$ has minimal distance to the point set $\{A(p_1),\ldots,A(p_m)\}$, 
i.e., the set of optimal lines is equivariant with respect to transformations
which do not change the distance.
In addition, the set of algebraically optimal lines is equivariant 
with respect to scalings of the coordinates, 
i.e., $(x,y)\mapsto (\lambda x,\mu y)$ with $0\ne \lambda,\mu \in \R$, and
the set of geometrically optimal lines is equivariant with respect to dilations, 
i.e.~$(x,y)\mapsto (\lambda x,\lambda y)$ with $0\ne \lambda \in \R$.

It is useful to consider the symmetries of the point set in order to
find the appropriate definition of a distance for a specific application, 
to restrict the domain of definition of $f$, 
or to simplify determination of the optimal line for strictly convex functions $f$, 
since the optimal line is unique in that case.
\subsection{Point sets with multiplicities}
In this article we work with the standard assumption that the given points $p_1,\ldots,p_m$ are pairwise
distinct, i.e., $p_j\ne p_k$ for all $j\ne k$, and $m\geq 2$.
Investigating the algebraic distances we additionally assume that the $x$-coordinates
$x_1,\ldots x_m$ of the points $p_j=(x_j,y_j)^T$ are pairwise distinct. 
It is easier to formulate the results in this context.

For each distance considered  here we discuss how the main statement changes 
if the standard assumption is omitted.
Therefore, our set up covers also distances which come from weighted norm on $(d_1,\ldots,d_m)$, 
e.g., $\sum_{j=1}^m n_jd_j^p$ with $n_j\in\N$.
\section{Method of least squares - $L^2$-norm}
\label{sec:L2}
In this section we minimize the $L^2$-norm of $(d_1,\ldots,d_m)$.
We consider the function $f(g)= \sum_{j=1}^m d_j^2$.
The function $f$ is a quadratic polynomial in the parameters of a line 
for the algebraic and the geometric distance.
Using differential calculus we derive explicit formulas of the optimal lines.
\subsection{Linear regression - minimal algebraic $L^2$-distance}
\label{sec:LineareRegression}
Given points $p_j=(x_j,y_j)^T\in \R^2$, $j=1,\ldots,m$, we determine a linear function
$y(x)=ax+b$, $a,b\in\R$, with minimal algebraic $L^2$-distance to the point  set  $\{Êp_1,\ldots,p_m\}$, i.e., the minimum of the function $f:\R^2\to\R$ defined by  
$$f(a,b) := \sum_{j=1}^n (y_j-y(x_j))^2 = \sum_{j=1}^m (y_j-ax_j-b)^2.$$
This approach excludes lines which are parallel to the $y$-axis ($x=c$).
Therefore, we assume $x_j\ne x_k$ for all $j\ne k$.
In particular, the points $p_1\ldots,p_m$ are pairwise distinct.
\subsubsection{Critical points}
The function $f$ is differentiable and
\begin{align*}
df & = (f_a,f_b) = \left( -2\sum_{j=1}^m (y_j-ax_j-b)x_j,  -2\sum_{j=1}^m (y_j-ax_j-b)\right) \\
& = 2\left( -\sum_{j=1}^m x_jy_j+a\sum_{j=1}^mx_j^2+b\sum_{j=1}^mx_j,  
	-\sum_{j=1}^m y_j+a\sum_{j=1}^m x_j + mb\right) .
\end{align*}
The critical points of $f$ are given by the solution of the following system of equations:
\begin{equation}
\label{eq:kritischer Punkt-lineare Regression}
a\sum_{j=1}^mx_j^2+b\sum_{j=1}^mx_j = \sum_{j=1}^m x_jy_j , \quad
	a\sum_{j=1}^m x_j + bm = \sum_{j=1}^m y_j
\end{equation}
\subsubsection{Convexity}
The function $f$ is convex, because the summands $(ax_j+b-y_j)^2$ are convex functions. 
These summands are strictly convex if and only if $x_j\ne 0$.
Since the inequality $x_j\ne 0$  holds for at least one index  $j$, the function $f$ is strictly convex.
This can be checked directly using the Hesse-matrix $H_f$:
$$H_f =  \begin{pmatrix}\sum_{j=1}^mx_j^2 & \sum_{j=1}^mx_j \\ \sum_{j=1}^mx_j & m \end{pmatrix},
	\quad  \det(H_f) = m\sum_{j=1}^mx_j^2-\left(\sum_{j=1}^mx_j \right)^2$$
It follows from the Schwarz inequality that 
$$\left(\sum_{j=1}^m x_j\right)^2 =\left(\sum_{j=1}^m x_j\cdot 1\right)^2 
	\leq \left( \sum_{j=1}^m x_j^2\right)\left(\sum_{j=1}^m 1^2\right) = m \sum_{j=1}^mx_j^2,$$
where equality holds if and only if $\vec{x}=(x_1,\ldots,x_m)$ is a multiple of the vector $(1,\ldots,1)$. 
Since $x_1,\ldots,x_m$ are pairwise distinct, we have $\sum_{j=1}^mx_j^2 >0$ and  $\det(H_f) >0$.
The matrix $H_f$ is strictly positive definite at all points $(a,b)\in\R^2$. 

There exists exactly one critical point, since the function $f$ is strictly convex.
The function $f$ attains its local and global minimum at this unique critical point.
\subsubsection{Globale minimum}
The solution of the system of linear equations  (\ref{eq:kritischer Punkt-lineare Regression}) 
gives the global minimum of the function $f:\R^2\to\R$. 
The second equation means that the optimal line contains the center of mass of the points
$p_1,\ldots,p_m$
\begin{equation}
\bar{p}Ê:= \frac{1}{m} \sum_{j=1}^m p_j = (\bar{x},\bar{y})^T \text{ with }
	\bar{x} = \frac{1}{m}\sum_{j=1}^m x_j,\, \bar{y} = \frac{1}{m}\sum_{j=1}^m y_j.
\end{equation}
Thus $b = \bar{y}-a\bar{x}$.

Moving the origin to the center of mass we obtain the new coordinates
$\tilde{x}:=x-\bar{x}$, $\tilde{y} := y-\bar{y}$, $\tilde{x}_j:=x_j-\bar{x}$ and $\tilde{y}_j:=y_j-\bar{y}$.
System (\ref{eq:kritischer Punkt-lineare Regression}) transforms into the equations
$\tilde{a}\sum_{j=1}^m \tilde{x}_j^2 = \sum_{j=1}^m \tilde{x}_j\tilde{y}_j$ and  $\tilde{b} = 0$ 
for the optimal line $\tilde{y} = \tilde{a}\tilde{x}+\tilde{b}$, 
since $\sum_{j=1}^m \tilde{x}_j = \sum_{j=1}^m \tilde{y}_j = 0$.
Set
\begin{equation}
S_{xx} := \sum_{j=1}^m \tilde{x}_j^2 = \sum_{j=1}^m (x_j-\bar{x})^2, \, 
S_{xy} := \sum_{j=1}^m \tilde{x}_j\tilde{y}_j = \sum_{j=1}^m (x_j-\bar{x})(y_j-\bar{y}).
\end{equation}
\begin{theorem}
\label{satz:eindeutigeRegressionsgerade}
Let $p_j=(x_j,y_j)^T$, $j=1,\ldots,m$,  such that $x_j\ne x_k$ for all $j\ne k$.\\
There exists a unique linear function $y(x) = ax+b$ with minimal algebraic $L^2$-distance to the set
$\{Êp_1,\ldots,p_m\}$. 
This function has slope $S_{xy}/S_{xx}$ and its graph
contains the center of mass of the set $\{ p_1,\ldots, p_m\}$, i.e.,
\begin{equation}
y(x) = \frac{S_{xy}}{S_{xx}}(x-\bar{x}) + \bar{y}.
\end{equation}
\end{theorem}
\begin{proof}
The translation $(x,y)\mapsto(\tilde{x},\tilde{y})$ does not change the slope.
\end{proof}
\begin{cor}
\label{folg:algL2Spiegelsymmetrie}
If the set $\{ p_1,\ldots,p_m\}$ is invariant under the reflection in the line $x=\bar{x}$, then $S_{xy}=0$.
\end{cor}
\begin{proof}
Since the optimal linear function is unique, it is invariant under the reflection in the line $x=\bar{x}$.
The graph of the optimal function is not the line $x=\bar{x}$. Thus, it is perpendicular to $x=\bar{x}$. 
\end{proof}
\begin{rem}
Theorem \ref{satz:eindeutigeRegressionsgerade} remains true 
if the condition $x_j\ne x_k$ for all $j\ne k$ is weakened to 
there existence of  indices $j\ne k$ satisfying $x_j\ne x_k$.
However, if $x_1=\ldots = x_m = \bar{x}$, then the graph of any linear function 
$y(x) = a(x-\bar{x})+\bar{y}$ with $a\in\R$ is an algebraically $L^2$-optimal line.
\end{rem}
\subsubsection{Generalizations}
Similarly, there exists a unique linear function $y(x_1,\ldots,x_n) = \sum_{i=1}^n a_ix_i+a_0$
having minimal algebraic $L^2$-distance to a given finite set in $\R^{n+1}$. 
Whenever a finite set in the plane has to be approximated by functions 
which are linear in the parameters to be optimized, the method that worked in this section 
can be applied, because the critical points of the convex objective function are the solution of a linear system of equations.
\subsection{Minimal geometric $L^2$-distance}
\label{sec:Euklidisch-quadratischesMittel}
Given pairwise distinct points $p_j\in \R^2$, $j=1,\ldots,m$, we determine lines
$g\subset \R^2$ with minimal geometric $L^2$-distance to the point set $\{Êp_1,\ldots,p_m\}$, 
i.e., the minimum of the function $f:\R\times S^1\to\R$ defined by
$$f(c,n) := \sum_{j=1}^n (c-\langle p_j,n\rangle)^2 .$$
Here, a line $g\subset\R^2$ is described by one of its normal vectors $n\in S^1$ and a point $q_0\in g$.
The points of the line $g$ are the solutions of the equation 
$\langle q,n\rangle = \langle q_0,n\rangle =: c \in \R$, i.e., $g=\{ q\in \R^2 : \langle q,n\rangle = c\}$.
The geometric distance between $p_j$ and  $g$ is given by $|c-\langle p_j,n\rangle|$ 
(see subsection \ref{subsubsec:Formel-Lot-Punkt-Gerade}).
\subsubsection{Geometric distance between a point and a line}
\label{subsubsec:Formel-Lot-Punkt-Gerade}
For $q_j\in g=\{ q\in \R^2 : \langle q,n\rangle = c\}$ the equation 
$d_j = \min_{q\in g} \| q-p_j\| = \| q_j-p_j\|$ holds if and only if $q_j - p_j \perp g$.
In the plane, the vector $p_j-q_j$ is perpendicular to the line $g$ if and only if  the vector $p_j-q_j$ 
is a multiple of the normal vector $n$ of the line, i.e., $q_j-p_j = \lambda n$. 
In that case, $q_j=p_j+\lambda n$ and $d(p_j,g)=|\lambda|$.
Now,
$$c= \langle q_j,n\rangle =  \langle p_j+\lambda n,n\rangle= \langle p_j,n\rangle +\lambda \|n\|^2
	= \langle p_j,n\rangle +\lambda,$$
since  $q_j=p_j+\lambda n$ and $q_j\in g$.
Thus, $\lambda = c-\langle p_j,n\rangle = \langle q_0-p_j,n\rangle$.
\subsubsection{Minimum for fixed normal vector}
We fix a normal vector $n\in S^1$ and define $f(c):=f(c,n)$. Now
$$ f'(c) = 2\sum_{j=1}^m (c-\langle p_j,n\rangle) 
	= 2mc - 2\left\langle \sum_{j=1}^m p_j,n\right\rangle, \quad  f''(c)  \equiv  2m >0.$$
Thus, the function $f:\R\to\R$, $c\mapsto f(c)$, is strictly convex.
Every local extremum is a global minimum.
The equation $f'(c)=0$ holds if and only if
\begin{equation}
c = \left\langle \frac{1}{m}\sum_{j=1}^m p_j,n\right\rangle = \langle \bar{p},n \rangle
	\text{ where } \bar{p} = \frac{1}{m} \sum_{j=1}^m p_j.
\end{equation}
For any fixed normal vector $n$ the line which contains the center of mass $\bar{p}$ gives the
minimum of $f:\R\to\R$, $c\mapsto f(c)$.
\subsubsection{Elimination of the variable $c$}
In order to find the minimum of the function $f:\R^2\to\R$, $(c,n)\mapsto f(c,n)$, 
it is sufficient to investigate lines that contain the center of mass $\bar{p}$.
We consider $f:S^1\to\R^{\geq 0}$ defined by
$$n\mapsto f(\langle \bar{p},n\rangle,n )  = \sum_{j=1}^m \langle \bar{p}-p_j,n\rangle^2
	 = \sum_{j=1}^m \langle \tilde{p}_j,n\rangle^2
	 = \sum_{j=1}^m (\tilde{p}_j^T n)^2 \text{ where } \tilde{p}_j=\bar{p}-p_j.$$
The change of notation (from  $p_j$ to $\tilde{p}_j= (\tilde{x}_j,\tilde{y}_j)^T$) 
corresponds to the translation of the origin to the center of mass $\bar{p}$.
The function $f(n)$ is even, i.e., $f(n) = f(-n)$ for all $n\in S^1$.
\subsubsection{Optimal normal vector}
\label{subsubsec:geo-l2-Drehung}
We use the parametrization $n(t) = (\cos t, \sin t)^T$ and consider the function $h(t) :=f(n(t))$.
Since $\tilde{p}_j^Tn = n^T \tilde{p}_j$ and 
$$n'(t)= \begin{pmatrix} -\sin t \\ \cos t \end{pmatrix} = Jn(t) \text{ with } 
	J=\begin{pmatrix} 0 & -1 \\ 1 & 0 \end{pmatrix},$$ 
we obtain
$$\frac{d h}{d t}= 2\sum_{j=1}^m \tilde{p}_j^T n  \tilde{p}_j^T \frac{d n}{d t}
	= 2\sum_{j=1}^m n^T \tilde{p}_j   \tilde{p}_j^T J n 
	= 2n^T\left(\sum_{j=1}^m \tilde{p}_j   \tilde{p}_j^T \right)J n = 2 n^T S J n  $$
with
\begin{equation}
S = \sum_{j=1}^m \tilde{p}_j   \tilde{p}_j^T 
= \begin{pmatrix} S_{xx} & S_{xy} \\ S_{xy} & S_{yy} \end{pmatrix},\,
S_{xx} =  \sum_{j=1}^m \tilde{x}_j^2, \, S_{xy} = \sum_{j=1}^m \tilde{x}_j\tilde{y}_j ,\,
	S_{yy} =\sum_{j=1}^m \tilde{y}_j^2 .
\end{equation}
It is easy to check that the matrix $S$ is symmetric:
$$S^T = \left( \sum_{j=1}^m \tilde{p}_j   \tilde{p}_j^T\right)^T 
	= \sum_{j=1}^m \left(\tilde{p}_j   \tilde{p}_j^T\right)^T 
	= \sum_{j=1}^m \left(\tilde{p}_j^T\right)^T   \tilde{p}_j^T 
	= \sum_{j=1}^m \tilde{p}_j   \tilde{p}_j^T = S$$
The fact $Jn\perp n$ yields $h'(t)= 0 \Leftrightarrow n \perp SJn \Leftrightarrow SJn = \lambda Jn$ 
for a $\lambda\in\R$.
This means that $Jn$ is an eigenvector of the matrix $S$. 
The symmetric matrix $S$ is diagonalisible. The eigenspaces of $S$ are perpendicular to each other.
Thus, $n$ is an eigenvector of $S$ if and only if  $Jn$ is an eigenvector of $S$.
\begin{itemize}
\item If the matrix $S$ has two different eigenvalues $\lambda_1\ne \lambda_2$,
	then the associated normalized eigenvectors $\pm n_1$, $\pm n_2$ 
	are the critical points of $f:n\mapsto f(n)$.
\item If the matrix $S$ has a two dimensional eigenspace, then $h'(t) \equiv 0$.
	Thus all lines through the center of mass are optimal.
\end{itemize}
Using the identity $J^2=-\Id$ we derive
\begin{align*}
\frac{d^2 h}{d t^2}= & = \frac{d}{d t}\left( 2n^T S J n \right)
	 = 2 \left(  \left( SJn\right)^T \frac{d n}{d t}+ n^T SJ  \frac{d n}{d t}\right) \\
& = 2\left((Jn)^T S^T Jn + n^T SJJn\right) = 2\left( (Jn)^T S Jn-n^T S n\right).
\end{align*}
Now,
$$h''(\pm n_1)  = 2(\lambda_2-\lambda_1) \text{ and } h''(\pm n_2)  = 2(\lambda_1-\lambda_2).$$
Therefore, the normalized eigenvectors of the smallest and the largest eigenvalue correspond to the
local minima respectively maxima of $f(n)$.
These local extrema are global, since $f(n)$ is symmetric.
\begin{lemma}
Let $E_-$ be the eigenspace of  the smallest eigenvalue of $S$. 
\begin{itemize}
\item $E_-=\R^2$ if and only if  $S_{xy} = 0$ und $S_{xx}=S_{yy}$.
\item $(1,0)^T \in E_-$ if and only if $S_{xy} = 0$ und $S_{xx}\leq S_{yy}$.
\item If $(1,0)^T \not\in E_-$, then $E_- = \R(a,-1)^T$ with
$$a= \frac{2S_{xy} }{S_{xx}-S_{yy}+\sqrt{(S_{xx}-S_{yy})^2+4S_{xy}^2}}.$$
\end{itemize}
\end{lemma}
\begin{proof}
The vector $(1,0)^T$ is an eigenvector of $S$ if and only if $S$ is a diagonal matrix, i.e., $S_{xy}= 0$.
Then $(1,0)^T$ and $(0,1)^T$ are eigenvectors of the eigenvalues $S_{xx}$ respectively $S_{yy}$.

Let us calculate the eigenvalues of $S$:
\begin{align*}
0 & = \det(S-\lambda \Id) = \left( S_{xx}-\lambda\right) \left( S_{yy}- \lambda\right) - S_{xy}^2 \\
& = \lambda^2-\lambda \left(S_{xx}+S_{yy}\right) +  S_{xx}S_{yy} - S_{xy}^2 \\
\lambda & = \frac{1}{2}(S_{xx}+S_{yy}) \pm \frac{1}{2} D  \text{ with } 
	D =\sqrt{(S_{xx}-S_{yy})^2+4S_{xy}^2} 
\end{align*}
If $(1,0)^T\not\in E_-$, then $\dim E_- = 1$ and $S_{xy}\ne 0$ or  $S_{xx}>S_{yy}$. 
The general solution of the linear equation $(S-\lambda \Id) v= 0$ with 
$\lambda = (S_{xx}+S_{yy} - D)/2 $ is $v=t(a,-1)^T$ with $t\in \R$ and $a = 2S_{xy}/(S_{xx}-S_{yy} +D)$.
\end{proof}
\begin{theorem}
\label{satz:Formel-geo-L2-Optimum}
Let $p_j\in\R^2$, $j=1,\ldots,m$, be pairwise distinct points. \\
A line $g\subset \R^2$ has minimal geometric $L^2$-distance to the set $\{ p_1,\ldots,p_m\}$
if and only if 
$\bar{p}\in g$ and the eigenspace of the smallest eigenvalue of $S$ contains a normal vector of $g$.
\begin{itemize}
\item If $S_{xy} = 0$ and $S_{xx} = S_{yy}$, then any line through $\bar{p}$ is optimal.
\item If $S_{xy} = 0$ and $S_{xx} < S_{yy}$, then there exists a unique optimal line: $x=\bar{x}$
\item If $S_{xy} \ne 0$ or $S_{xx} > S_{yy}$, then there exists a unique optimal line:
	\begin{equation}
	y = \frac{2S_{xy}}{S_{xx}-S_{yy}+D}(x-\bar{x}) +\bar{y} 
	\end{equation}
\end{itemize}
\end{theorem}
\begin{proof}
An optimal line with normal vector $(a,-1)^T$ can be parametrized by $g=\{Ê\bar{p}+t(1,a)^T:t\in\R\}$.
The slope of this line is $a$.
\end{proof}
\begin{cor}
If the point set $\{ p_1,\ldots, p_m\}$ is invariant under reflection in a line $g\ni\bar{p}$,
then this line $g$ or the line perpendicular to $g$ containing $\bar{p}$ is a line with
minimal geometric $L^2$-distance to the set $\{ p_1,\ldots, p_m\}$.
\end{cor}
\begin{proof}
Let $l\ne g$ be an optimal line that contains $\bar{p}$ and is not perpendicular to $g$. 
Since $l$ is not invariant under the reflection in $g$, there exist at least two optimal lines.
Thus, all lines containing $\bar{p}$ are optimal.
\end{proof}
\begin{cor}
If the point set $\{ p_1,\ldots, p_m\}$ is invariant under a rotation 
around $\bar{p}$ through an angle $\phi \not\in \Z\pi$, then $S_{xy}=0$ and $S_{xx}=S_{yy}$.
\end{cor}
\begin{proof}
Since $\phi\ne k\pi$ for all $k\in\Z$, no line is invariant under the rotation. 
Thus, all lines containing $\bar{p}$ are optimal.
\end{proof}
\begin{rem}
Theorem \ref{satz:Formel-geo-L2-Optimum} and its corollaries remain true if the points $p_1,\ldots,p_m$ 
are not pairwise distinct.
\end{rem}
\subsubsection{Generalizations}
A natural generalization of the subject in this section is an affine subspace of fixed dimension $k<n$ 
with minimal geometric $L^2$-distance to a given finite set in $\R^n$.
The corresponding optimal lines and planes in $\R^3$ are discussed in \cite{P}.
For $k=n-1$ this problem leads to overdetermined systems of linear equations that can be treated with total least squares (TLS) techniques.
In this section the essential information about the point set is stored in the matrix $S$.
The decomposition of a similar matrix $S$ reappears in a method of multivariate statistics called
principal component analysis (PCA).
\subsection{Example an comparison algebraic versus geometric}
Any $L^2$-optimal line contains the center of mass of the set $\{ p_1,\ldots,p_m\}$. 
The algebraic problem always possesses a unique solution.
The geometric problem admits a unique solution that is a graph of a linear function if and only if
$S_{xy}\ne 0$ or $S_{xx}>S_{yy}$.
\begin{theorem}
The algebraic and the geometric $L^2$-optimal lines coincide if and only if
this line contains all points $p_j$ or  $S_{xy}=0$ and $S_{xx}\geq S_{yy}$.
\end{theorem}
\begin{proof}
If $S_{xy}=0$ and $S_{xx} > S_{yy}$, then the algebraic and the geometric $L^2$-optimal lines coincide 
if and only if
$$ \frac{S_{xy}}{S_{xx}} =  \frac{2S_{xy}}{S_{xx}-S_{yy}+D} .$$
This equation holds if and only if $S_{xy} = 0$ or 
\begin{align*}
S_{xx}-S_{yy}+D & = 2S_{xx} \\
D & = S_{xx}+S_{yy} \\
(S_{xx}-S_{yy})^2+4S_{xy} & = (S_{xx}+S_{yy})^2 \\
S_{xx}S_{yy} & = S_{xy}.
\end{align*}
It follows from the Schwarz inequality that
$$S_{xx}S_{yy} = \sum_{j=1}^m \tilde{x}_j^2 \sum_{j=1}^m \tilde{y}_j^2 \geq 
(\sum_{j=1}^m \tilde{x}_j\tilde{y}_j)^2 = S_{xy}^2,$$
where equality holds if and only if $\lambda \tilde{x}_j = \tilde{y}_j$ for a $\lambda\in\R$ and all $j$.
\end{proof}
We consider the four points
$p_1=(0,0)^T$, $p_2=(1,1)^T$, $p_3=(2,2)^T$ and $p_4=(3,3/2)^T$ (see Figure \ref{l2m4}).
Their center of mass is $\bar{p} = (3/2, 9/8)^T$. Now, $\tilde{p}_1=(-3/2,-9/8)^T$, 
$\tilde{p}_2=(-1/2,-1/8)^T$, $\tilde{p}_3=(1/2,7/8)^T$, $\tilde{p}_4 = (3/2,3/8)^T$, 
$S_{xx} = 5$, $S_{yy} = 35/16$ und $S_{xy} = 11/4$.
The algebraic $L^2$-optimal linear function is
$$y(x) = a\left(x-\frac{3}{2}\right) +\frac{9}{8} \text{ with } a= \frac{11}{20} = 0.55.$$
The geometric $L^2$-optimal line is the graph of the linear function
$$y(x) = a\left(x-\frac{3}{2}\right) +\frac{9}{8} \text{ with } a=\frac{88}{45+\sqrt{45^2+88^2}}\approx 0.61.$$
The algebraic and the geometric $L^2$-optimal lines intersect at $\bar{p}$.
\begin{figure}
\begin{center}
  \mbox{\scalebox{.3}{\includegraphics{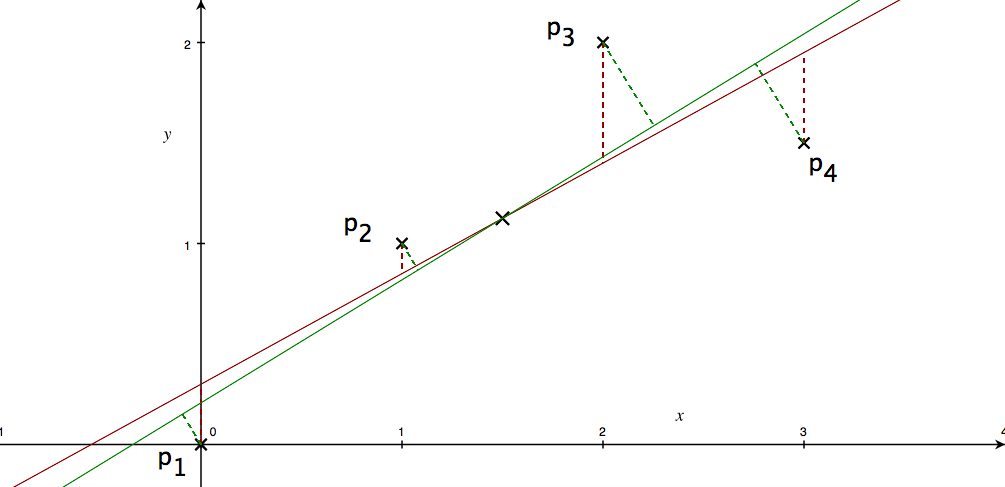}}}
\caption{vertical and orthogonal $L^2$-optimal lines}
\label{l2m4}
\end{center}
\end{figure}
\section{Absolute distance - $L^1$-Norm}
\label{sec:TotalerAbstand}
In this section we minimize the sum of the distances of the single points $p_j$.
We consider the function $f(g)=\sum_{j=1}^m d_j$.
This function is differentiable at $g$ if and only if 
$d_j>0$ for all $j$, i.e., $g$ contains none of the points $p_j$.
We prove that there exists a global minimum of $f$ for the algebraic and the geometric distance.
The set of all $L^1$-optimal lines is described. 
Explicit formulas of the elements of $M$ in terms of  $p_1,\ldots,p_m$ are not available,
yet the set $M$ can be completely characterized by comparing the values of $f$ at lines containing
at least two of the points $p_j$.
Minimizing the $L^1$-norm can be carried out in $\leq O(m^2)$ steps.
\subsection{Minimal algebraic $L^1$-distance}
\label{sec:Totaler-algebraischer-Abstand}
Given points $p_j=(x_j,y_j)^T\in \R^2$, $j=1,\ldots,m$, 
we determine all linear functions $y(x)=ax+b$, $a,b\in\R$,
with minimal algebraic $L^1$-distance to the point set $\{Êp_1,\ldots,p_m\}$, 
i.e., the minimum of the function $f:\R^2\to\R$ defined by
$$f(a,b)=\sum_{j=1}^m | y_j-(ax_j+b)|.$$
As in subsection \ref{sec:LineareRegression} we additionally assume that 
$x_j\ne x_k$ for all $j\ne k$.

The function $f$ is continuous, piecewise linear, convex and bounded from below.
Thus, $f$ admits a global minimum. The set of algebraically $L^1$-optimal lines $M$ is convex.
Since $f$ is not strictly convex, the set $M$ could be unbounded.
We show that $M$ is the convex hull of finitely many points.
\subsubsection{Decomposition of the index set}
\label{subsubsec:ZerlegungIndexmenge}
For every linear function $y(x)=ax+b$ we define
$J_+:=\{Êj:y_j > y(x_j)\}$, $J_0:=\{Êj:y_j = y(x_j)\}$ und $J_-:=\{Êj:y_j < y(x_j)\}$.
This decomposition depends on the parameters $a$ and $b$ (see Figure \ref{fig:ZerlegAlg}).
The sets $J_+$, $J_0$ and $J_-$ are pairwise disjoint,
$J_+ \cup J_0 \cup J_- = \{Ê1,\ldots,m\}$ and
\begin{align*}
f(a,b) & = \sum_{j\in J_+} (y_j-ax_j-b) -\sum_{j\in J_-} (y_j-ax_j-b) \\
& = b(|J_-| - |J_+|) + \left(\sum_{j\in J_+} y_j -\sum_{j\in J_-} y_j\right) -
	a\left( \sum_{j\in J_+} x_j -\sum_{j\in J_-} x_j\right).
\end{align*}
\begin{figure}
\noindent
\begin{minipage}[t]{.48\linewidth}
  \mbox{\scalebox{.26}{\includegraphics{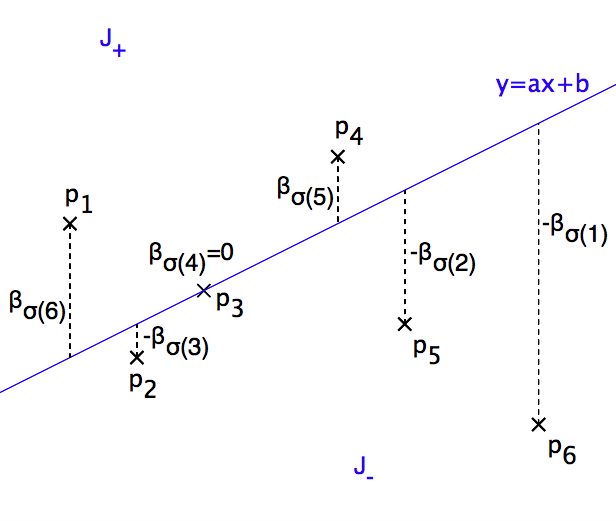}}}
\caption{algebraic decomposition of the index set, definition of $\beta_{\sigma(j)}$}
\label{fig:ZerlegAlg}
\end{minipage}
\hspace*{.02\linewidth}
\begin{minipage}[t]{.48\linewidth}
  \mbox{\scalebox{.26}{\includegraphics{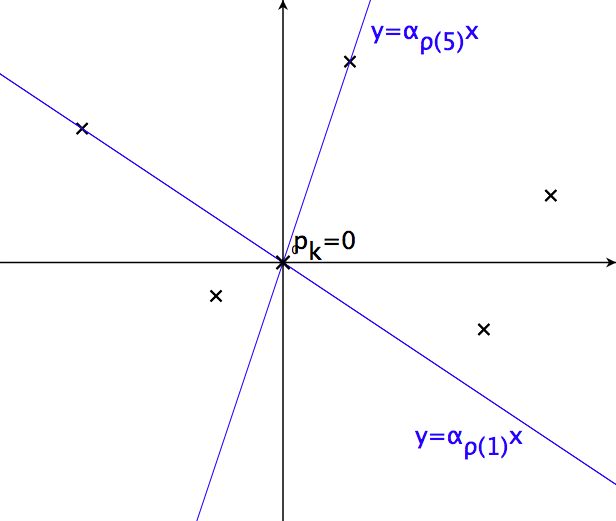}}}
\caption{definition of $\alpha_{\sigma(j)}$ for lines through $p_k$}
\label{fig:Drehalg}
\end{minipage}
\end{figure}
\subsubsection{Translation}
\label{subsubsec:Parallel-l1-alg}
For any $a\in\R$ we consider the values of $f$ at lines with fixed slope $a$ and variable $b$.
We want to determine the optimal line with slope $a$, i.e., $\min\{ f(a,b):b\in\R\}$.

Note that $f(a,0)= \sum_{j=1}^m |y_j-ax_j|$. Set $\beta_j:=y_j-ax_j$.
Let $\sigma$ be a permutation satisfying $\beta_{\sigma(1)}\leq \ldots \leq \beta_{\sigma(m)}$
(see Figure \ref{fig:ZerlegAlg}).
\begin{lemma}
For every $a\in\R$ the function $f(a,\cdot):\R\to\R$, $b\mapsto f(a,b)$ admits a global minimum.
If $f(a,b_0) = \min\{ f(a,b):b\in\R\}$, then $\beta_{\sigma(1)} \leq b_0 \leq  \beta_{\sigma(m)} $.
\end{lemma}
\begin{proof}
If $b>\beta_{\sigma(m)}$, then $J_+=J_0=\emptyset$.
If $b<\beta_{\sigma(1)}$, then $J_-=J_0=\emptyset$.
Consequently,
$$f(a,b) = mb-\sum_{j=1}^m y_j+a\sum_{j=1}^m x_j \text{ for } b>\beta_{\sigma(m)}$$
and
$$f(a,b) = -mb+\sum_{j=1}^m y_j-a\sum_{j=1}^m x_j \text{ for }b<\beta_{\sigma(1)}.$$
The function $f(b)$ is strictly decreasing for $b<\beta_{\sigma(1)}$ 
and strictly increasing for $b>\beta_{\sigma(m)}$. 
Continuity of $f$ yields
$$\inf_{b\in \R} f(a,b) = \inf_{\beta_\sigma(1) \leq b \leq \beta_{\sigma(m)}}  f(a,b)
= \min_{\beta_{\sigma(1)}\leq b \leq \beta_{\sigma(m)}} f(a,b) .$$
\end{proof}
\begin{lemma}
For every $a\in\R$ the function $f(a,\cdot)$ admits its global minimum at $b$ if and only if
$|J_-|+|J_0|- |J_+|\geq 0$ und $|J_+|+|J_0|- |J_-|\geq 0$.

For every $a\in\R$ the function $f(a,\cdot)$ admits its strict global minimum at $b$ if and only if
$|J_-|+|J_0|- |J_+| > 0$ und $|J_+|+|J_0|- |J_-| > 0$.
\end{lemma}
\begin{proof}
Let  $J_+$,  $J_0$, $J_-$ be the decomposition of the index set of the function $y(x) =ax+b$. 
If $\max\{ \beta_j-b: j\in J_- \} <\varepsilon < \min\{\beta_j-b: j\in J_+\} $, then the decomposition of the index set of the function $ax+b+\varepsilon$ is equal to the decomposition of $ax+b$.
Now
\begin{align*}
f(a,b+\varepsilon) & = \sum_{j\in J_+} (y_j-ax_j-b-\varepsilon) -\sum_{j\in J_-} (y_j-ax_j-b-\varepsilon)
	+\sum_{j\in J_0} | \varepsilon | \\
& = f(a,b) +\varepsilon \cdot 
	\begin{cases} |J_-|+|J_0|- |J_+| & \text{if }\varepsilon > 0 \\  
	|J_-|-|J_0|- |J_+| & \text{if } \varepsilon < 0 \end{cases}.
\end{align*}
\end{proof}
\begin{cor}
\label{folg:alg-l1-Zerlegung-im-Optimum}
For every $a\in\R$ the function $f(a,\cdot)$ admits its global minimum at $b$ if and only if
$| |J_+| - |J_-| | \leq |J_0|$. \\
Furthermore, $\min\{f(a,b) : b\in\R\} = f(a,b_0)$ if and only if $\beta_k \leq b_0 \leq \beta_l$ with
$$k=\left\lfloor \frac{m+1}{2}\right\rfloor,\, l = \left\lceil \frac{m+1}{2}\right\rceil.$$
\end{cor}
\begin{proof}
The inequalities $|J_-|+|J_0|- |J_+|\geq 0$ and $|J_+|+|J_0|- |J_-|\geq 0$ are both satisfied if and only if
$|J_0|\geq ||J_+|-|J_-||$.

If $ \beta_k \leq b \leq \beta_l$, then
\begin{itemize}
\item $|J_-|+|J_0| \geq k$, $|J_+|\leq m-k \leq |J_-|+|J_0|$, since $2k\geq m$.
\item $|J_+|+|J_0| \geq m-l+1$, $|J_-|\leq l-1 \leq |J_+|+|J_0|$, since $2(l-1)\leq m$.
\end{itemize}

If $b<\beta_k$, then $|J_-|+|J_0|\leq k-1$ and $|J_+|\geq m-(k-1) >|J_-|+|J_0|$, since $2(k-1)<m$.

If $b>\beta_l$, then $|J_+|+|J_0|\leq m-l$ and $|J_-|\geq l >|J_+|+|J_0|$, since $2l>m$.
\end{proof}
\begin{cor}
\label{folg:alg-l1-Verschiebung-in-Punkt}
Among all optimal lines with fixed slope $a\in \R$ 
there exists at least one that contains one of the points $p_j$.
\end{cor}
\begin{proof}
If $b=\beta_k$ or $b=\beta_l$, then $J_0\ne \emptyset$.
\end{proof}
\subsubsection{Rotation}
\label{subsubsec:Drehung-l1-alg}
For every point $p_k$ we consider the values of $f$ at lines containing $p_k$.
These lines are obtained by rotating one of them around $p_k$.
The condition $y(x_k)= y_k$ implies $b=y_k-ax_k$. 
For any fixed index $k$ we investigate the function $f:\R\to\R$ defined by
\begin{align*}
f(a) & :=f(a,y_k-ax_k) = \sum_{j=1}^m |y_j-ax_j-(y_k-ax_k)| \\
& = \sum_{j=1}^m |(y_j-y_k)-a(x_j-x_k)| = \sum_{j\ne k} |\tilde{y}_j-a\tilde{x}_j| 
\end{align*}
with $ \tilde{y}_j:=y_j-y_k$ and $\tilde{x}_j := x_j-x_k$.

The coordinate change $(x_j,y_j)\to (\tilde{x}_j,\tilde{y}_j)$ corresponds to the translation of the origin
to the center of the rotation.
The decomposition of the index set is now given by
$J_-=\{ j:\tilde{y}_j<a\tilde{x}_j\}$, $J_0=\{ j:\tilde{y}_j=a\tilde{x}_j\} \ni k$ and 
$J_+=\{ j:\tilde{y}_j>a\tilde{x}_j\}$. It follows
$$f(a) = \sum_{j\in J_+}\tilde{y}_j-\sum_{j\in J_-} \tilde{y}_j 
	- a \left(  \sum_{j\in J_+}\tilde{x}_j-\sum_{j\in J_-} \tilde{x}_j\right).$$
Set $\alpha_j:=\tilde{y}_j/\tilde{x}_j$ for $j\ne k$.
Let $\tau$ be a permutation satisfying $\tau(m)=k$ and 
$\alpha_{\tau(1)}\leq \ldots \leq \alpha_{\tau(m-1)}$ (see Figure \ref{fig:Drehalg}).
\begin{lemma}
\label{lem:alg-L1-Drehintervall}
The function $f:a\mapsto f(a)$ admits a global minimum.\\
If $f(a_0) =  \min\{ f(a):a\in\R\}$, then $\alpha_{\tau(1)} \leq a_0 \leq \alpha_{\tau(m-1)}$.
\end{lemma}
\begin{proof}
If $a>\alpha_{\tau(m-1)}$, then $J_0=\{ k\}$, $J_+=\{ j:\tilde{x}_j<0\}$, $J_-=\{ j:\tilde{x}_j >0\}$ and
$f'(a) = \sum_{j\in J_-} \tilde{x}_j - \sum_{j\in J_+} \tilde{x}_j > 0.$

If $a<\alpha_{\tau(1)}$, then $J_0=\{ k\}$, $J_+=\{ j:\tilde{x}_j>0\}$, $J_-=\{ j:\tilde{x}_j <0\}$ and
$f'(a) = \sum_{j\in J_-} \tilde{x}_j - \sum_{j\in J_+} \tilde{x}_j > 0.$

Continuity of $f$ yields
$$\inf_{a\in\R} f(a) = \inf_{\alpha_{\tau(1)}\leq a \leq \alpha_{\tau(m-1)}} f(a) 
= \min_{\alpha_{\tau(1)}\leq a \leq \alpha_{\tau(m-1)}} f(a).$$
\end{proof}
\begin{cor}
The function $f:\R^2\to\R$ defined by $f(a,b)=\sum_{j=1}^m |y_j-ax_j-b|$ admits a global minimum.
This global minimum is attained at a bounded set.
\end{cor}
\begin{proof}
If $f(a,b)$ is a global minimum of $f$, then $a_0\leq a \leq a_1$ and $b_0\leq b\leq b_1$ where
$$a_0 =\min_{j\ne k}\frac{y_j-y_k}{x_j-x_k} , \,
	a_1=\max_{j\ne k} \frac{y_j-y_k}{x_j-x_k},\,
	b_0 = \min_{i,j} y_j-a_ix_j, \, b_1= \max_{i,j}Êy_j-a_ix_j.$$
\end{proof}
\begin{lemma}
\label{lem:alg-l1-Drehung}
For every index $k$ holds $\min\{ f(a):a\in\R\} = \min\{Êf(\alpha_j) : j\ne k\}$. 
\end{lemma}
\begin{proof}
If $J_0=\{ k\}$, then $f$ is differentiable at the line corresponding to that decomposition. 
Note that  $J_0=\{ k\}$ if and only if $a\ne \alpha_j$ for all $j\ne k$. 
If $f'(a) = 0$ for $a\in\R$ with $a_{\sigma(k)}<a<a_{\sigma(k+1)}$, then $f$ is constant on the interval
$[a_{\sigma(k)},a_{\sigma(k+1)}]$.
\end{proof}
\subsubsection{The set of optimal lines}
Corollary \ref{folg:alg-l1-Verschiebung-in-Punkt} and Lemma \ref{lem:alg-l1-Drehung} imply the
existence of a line containing at least two  $p_j\ne p_k$ 
among all lines with minimal algebraic $L^1$-distance to the set $\{ p_1,\ldots,p_m\}$.
There are $\leq m(m-1)/2$ lines with this property.
It is sufficient to check the lines that additionally satisfy $| |J_+| - |J_-| | \leq |J_0|$ 
to find the minimum of $f$.

The line which contains the two points $p_j$ and $p_k$ with $j\ne k$ is called $g_{jk}$.
The line $g_{jk}$ is the graph of the linear function $y(x) = a_{jk}x+b_{jk}$ with
\begin{equation}
a_{jk} = \frac{y_j-y_k}{x_j-x_k} \text{ and } b_{jk} = \frac{y_kx_j-y_jx_k}{x_j-x_k}.
\end{equation}
The line $g_{jk}$ corresponds to a point in the domain of definition $\{Ê(a,b):a,b\in\R\}$ 
of the function $f$. This point $(a,b)$ is  the unique solution of the two equations $y_j=ax_j+b$ and $y_k=ax_k+b$, since $x_j\ne x_k$. 
We denote this point of intersection by $g_{jk}$ too, i.e., $g_{jk}=(a_{jk},b_{jk})$.
Let $E$ be the set of optimal lines containing at least two of the points $p_j$.
More precisely,
\begin{equation}
E:=\{ g_{jk}: j>k \text{ and } f(a_{jk},b_{jk}) = \min\{ f(a,b):a,b\in\R^2\} \}.
\end{equation}
\begin{theorem}
\label{satz:alg-L1-Optimenge}
Let $p_j=(x_j,y_j)^T$, $j=1,\ldots,m$,  such that $x_j\ne x_k$ for all $j\ne k$.
A linear function $y(x)=ax+b$ has minimal algebraic $L^1$-distance to the set $\{ p_1,\ldots,p_m\}$
if and only if $(a,b)$ is contained in the convex hull of $E$.
\end{theorem} 
\begin{proof}
If  $f(a_0,b_0)$ is the global minimum of $f$ and $y_k=a_0x_k+b_0$ for exactly one index 
$1\leq k\leq m$, then Lemma \ref{lem:alg-l1-Drehung} implies the existence of indices 
$1\leq j_1,j_2\leq m$ satisfying $j_1,j_2\ne k$ and $f(g_{j_1k})=f(g_{j_2k})=f(a_0,b_0)$. 
Consequently, the point $(a_0,b_0)$ lies on the line $y_k=ax_k+b$ between the points
$g_{j_1k}, g_{j_2k} \in E$.

If $f(a_0,b_0)$ is the global minimum of $f$ and $y_j\ne a_0x_j+b_0$ for all $j$,
then Corollary \ref{folg:alg-l1-Zerlegung-im-Optimum} implies the existence if indices $k,l$ 
satisfying $f(a_0,b_0) = f(a_0,b)$ for all $b_k:=y_k-a_0x_k\leq b \leq y_l-a_0y_l=:b_l$.
Since $f(a_0,b_k)=f(a_0,b_l)$ is the global minimum of $f$ and the lines corresponding to 
$(a_0,b_k)$ and $(a_0,b_l)$ contain the point $p_k$ respectively $p_l$, the points
$(a_0,b_k)$ and $(a_0,b_l)$ are contained in the convex hull of $E$. 
Now $(a_0,b_0)$ is a point of the line segment from $(a_0,b_k)$ to $(a_0,b_l)$.
Therefore, $(a_0,b_0)$ is in the convex hull of $E$. 
\end{proof}
\begin{rem}
With a small change in the definition of $E$ 
Theorem \ref{satz:alg-L1-Optimenge}  remains true, if the condition $x_j\ne x_k$ for all $j\ne k$ 
is weakened to existence of indices $j\ne k$ satisfying $x_j\ne x_k$.
This weaker assumption is sufficient to bound  interesting slopes using 
Lemma \ref{lem:alg-L1-Drehintervall}, because $\alpha_j \ne \pm\infty$ for at least on index $j$.
To adjust the definition of $E$ we then regard only lines $g_{jk}$ with $x_j\ne x_k$.

However, if $x_1=\ldots=x_m=\bar{x}$, then any graph of a linear function contains at most one of 
the points $p_j$. This would mean that $E=\emptyset$.
But note that the algebraic distances between the points $p_j$ and the linear function 
$y(x) = a(x-\bar{x}) + b$ are independent of  the variable $a$. 
We substitute the set $E$ by the set $B:=\{ y_k: f(0,y_k) = \min\{ f(0,y_j):j=1,\ldots,m\}\} \subset \R$.
Now, $y(x)=a(x-\bar{x})+b$ is an algebraically $L^1$-optimal linear function
if and only if $b$ is contained in the convex hull of $B$.
\end{rem}
\subsubsection{Examples}
\label{subsubsec:Beispiele-alg-l1}
\paragraph{Three points:}
We show the existence of a unique linear function with minimal algebraic $L^1$-distance to three given
pairwise distinct points.
Let $x_1<x_2<x_3$ and $y_1,y_2,y_3\in\R$.
There exists a line containing all three points $(x_j,y_j)^T$ if and only if
$(y_2-y_1)/(x_2-x_1) = (y_3-y_1)/(x_3-x_1)$.

If $(y_2-y_1)/(x_2-x_1) \ne (y_3-y_1)/(x_3-x_1)$, then
\begin{align*}
f(g_{31}) & = \left| y_2-\frac{y_3-y_1}{x_3-x_1}(x_2-x_1)-y_1\right| \\
&  = \frac{|(y_2-y_1)(x_3-x_1)-(y_3-y_1)(x_2-x_1)|}{x_3-x_1} \\
	& = \frac{|(y_2-y_1)(x_3-x_2)-(y_3-y_2)(x_2-x_1)|}{x_3-x_1} \\
f(g_{21}) & = \left| y_3-\frac{y_2-y_1}{x_2-x_1}(x_3-x_1)-y_1\right| \\
	& = \frac{|(y_3-y_1)(x_2-x_1)-(y_2-y_1)(x_3-x_1)|}{x_2-x_1} >  f(g_{31}) \\
f(g_{32}) & = \left| y_1-\frac{y_3-y_2}{x_3-x_2}(x_1-x_2)-y_2\right| \\
	& = \frac{|(y_1-y_2)(x_3-x_2)-(y_3-y_2)(x_1-x_2)|}{x_3-x_2} >  f(g_{31})
\end{align*} 
because $x_3-x_1 > x_2-x_1$ and $x_3-x_1>x_3-x_2$.
There is a unique line with minimal algebraic $L^1$-distance to $\{p_1,p_2,p_3\}$. 
It is the line through $p_1$ and $p_3$.
\paragraph{Family of optimal lines for $m=4$:}
We consider the points $p_1=(0,0)^T$, $p_2=(1,1)^T$, $p_3=(2,2)^T$ and $p_4=(3,3/2)^T$ 
(see Figure \ref{fig:algl1m4}). Note that $g_{21} = g_{31} = g_{32}$, 
because the three points $p_1$, $p_2$ and $p_3$ are collinear.
Now $f(g_{21})=3/2$, $f(g_{43}) = 9/2$, $f(g_{42}) = 3/2$, $f(g_{41}) = 3/2$.
Thus, the set $M$ of optimal lines is the convex hull of
$g_{21} = g_{31} = g_{32}$, $g_{41}$ and $g_{42}$ (see Figure \ref{fig:algl1m4def}). 
\begin{figure}
\noindent
\begin{minipage}[t]{.49\linewidth}
  \mbox{\scalebox{.25}{\includegraphics{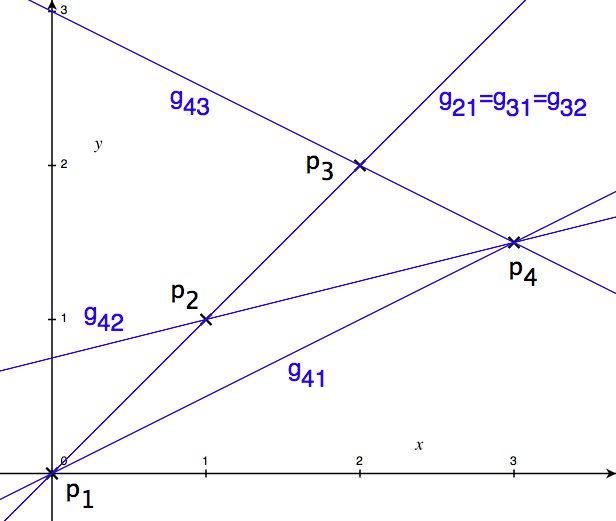}}}
\caption{}
\label{fig:algl1m4}
\end{minipage}
\begin{minipage}[t]{.50\linewidth}
  \mbox{\scalebox{.26}{\includegraphics{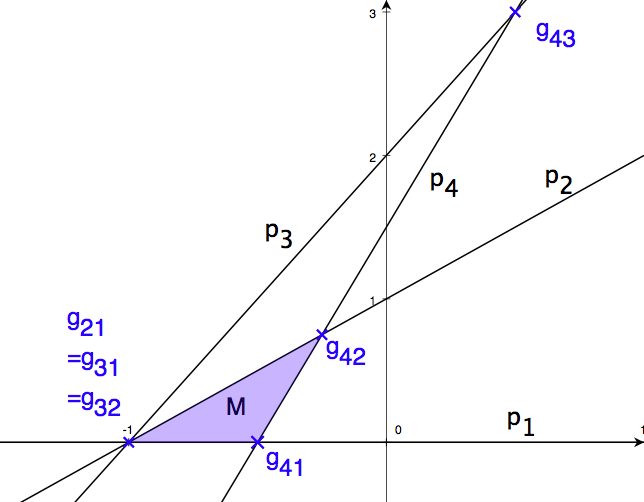}}}
\caption{}
\label{fig:algl1m4def}
\end{minipage}
\end{figure}
\paragraph{Invariance under reflection for $m=5$:}
Let us consider the five points $p_1=(-2,1)^T$, $p_2=(-1,-1)^T$, $p_3=(0,0)^T$, $p_4=(1,-1)^T$ and 
$p_5=(2,1)^T$ (see Figure \ref{fig:algl1m5}). 
This set is symmetric with respect to the reflection in the $y$-axis.
Hence, the set $M$ is invariant under this reflection.
The lines $g_{51}$, $g_{21}$, $g_{42}$ and $g_{54}$ appear as dotted lines 
in Figure \ref{fig:algl1m5}, because they do not satisfy the condition $| |J_+| - |J_-| | \leq |J_0|$.
Now $f(g_{31})= f(g_{53}) = 4$, $f(g_{32})= f(g_{43}) = 6$ and $f(g_{41})= f(g_{52}) = 13/3>4$.

The set $M$ of optimal lines is the line segment from $g_{53}$ to $g_{31}$ in Figure \ref{fig:algl1m5def}. 
The elements of $M$ correspond to the functions $y(x)= ax$ with $-1/2 \leq a \leq 1/2$.
\begin{figure}
\noindent
\begin{minipage}[t]{.49\linewidth}
  \mbox{\scalebox{.25}{\includegraphics{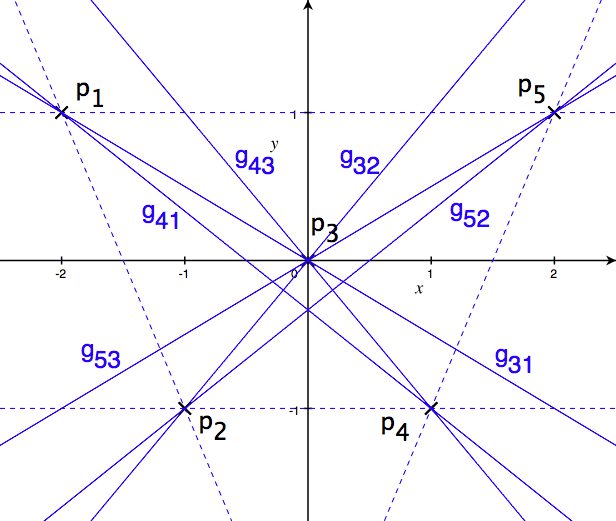}}}
\caption{}
\label{fig:algl1m5}
\end{minipage}
\begin{minipage}[t]{.50\linewidth}
  \mbox{\scalebox{.26}{\includegraphics{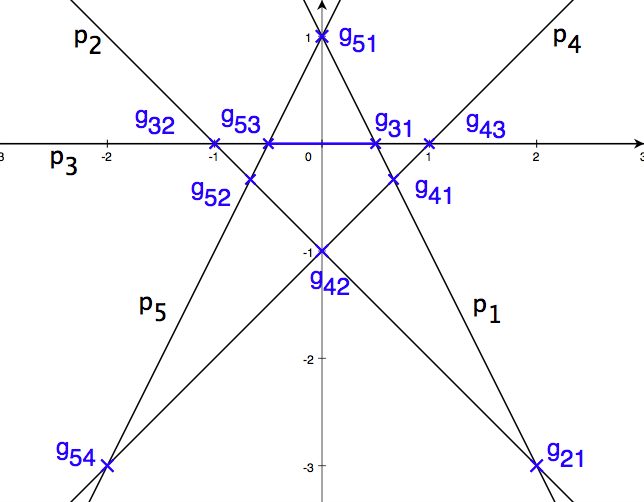}}}
\caption{}
\label{fig:algl1m5def}
\end{minipage}
\end{figure}
\subsection{Minimal geometric $L^1$-distance}
\label{subsec:geo-l1}
Given pairwise distinct points $p_j\in \R^2$, $j=1,\ldots,m$, we determine lines described as
$g = \{ q\in \R^2 : c=\langle q,n\rangle\}$, $c\in\R$, $n\in S^1$ with minimal geometric $L^1$-distance
to the point set $\{Êp_1,\ldots,p_m\}$, i.e., the minimum of the function
$f:\R\times S^1\to\R$ defined by
$$f(c,n)=\sum_{j=1}^m | c-\langle p_j,n\rangle| .$$
Similar to the definitions in subsection \ref{subsubsec:ZerlegungIndexmenge}
we decompose the index set.
For every line  $\{Êq\in\R^2: \langle q,n\rangle = c\}$ set
$J_+:=\{j: \langle p_j,n\rangle > c\}$, $J_0:=\{j: \langle p_j,n\rangle = c\}$ and
$J_-:=\{j: \langle p_j,n\rangle < c\}$ (see Figure \ref{fig:ZerlegGeo}).
Now
\begin{align*}
f(c,n) & = \sum_{j\in J_-} (c-\langle p_j,n\rangle) +\sum_{j\in J_+} (\langle p_j,n\rangle -c) \\
& = c(|J_-|-|J_+|) -  \sum_{j\in J_-}\langle p_j,n\rangle +\sum_{j\in J_+} \langle p_j,n\rangle.
\end{align*}
\subsubsection{Translation}
\label{subsubsec:Parallel-l1-geo}
For any $n\in S^1$ we restrict $f$ to lines with fixed normal vector $n$.
We want to determine the optimal line with normal vector $n$, i.e., $\min\{ f(c,n):c\in\R\}$.
Set $\gamma_j:=\langle p_j,n\rangle$ and let  $\sigma$ be a permutation satisfying 
$ \gamma_{\sigma(1)}\leq \ldots \leq \gamma_{\sigma(m)}$ (see Figure \ref{fig:ZerlegGeo}).
\begin{figure}
\noindent
\begin{minipage}[t]{.48\linewidth}
  \mbox{\scalebox{.26}{\includegraphics{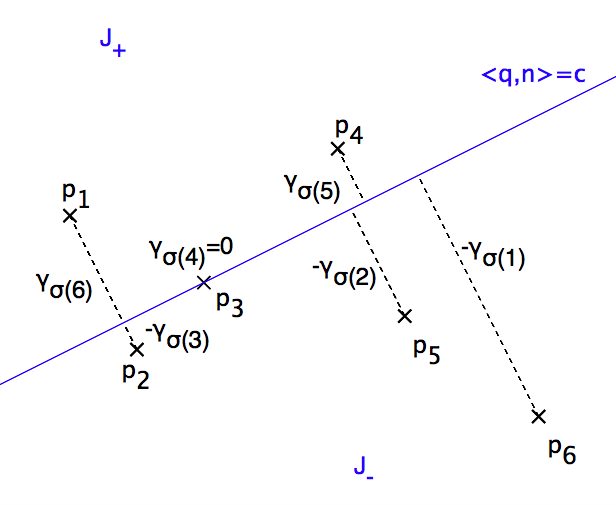}}}
\caption{geometric decomposition of the index set, definition of $\gamma_{\sigma(j)}$}
\label{fig:ZerlegGeo}
\end{minipage}
\hspace*{.02\linewidth}
\begin{minipage}[t]{.48\linewidth}
  \mbox{\scalebox{.26}{\includegraphics{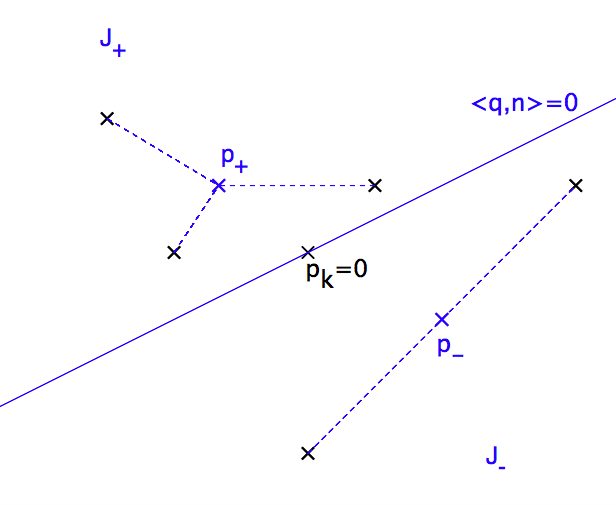}}}
\caption{definition of $p_+$ and $p_-$ for lines through  $p_k$}
\label{fig:DrehGeo}
\end{minipage}
\end{figure}

\begin{lemma}
For every $n\in S^1$ the function $f:\R\to\R$, $c\mapsto f(c,n)$, admits a global minimum. 
If $\min\{ f(c,n):c\in\R\} = f(c_0,n)$, then $\gamma_{\sigma(1)}\leq c_0\leq \gamma_{\sigma(m)}$.
\end{lemma}
\begin{proof}
The function $f:c\mapsto f(c,n)$ is continuous, piecewise linear and  convex.
If $J_0=\emptyset$, then $f$ is differentiable and $f'(c) = |J_-|-|J_+|$.

If $c>\gamma_{\sigma(m)}$, then $J_0=J_+=\emptyset$ and $f'(c)=m>0$.
If $c<\gamma_{\sigma(1)}$, then $J_0=J_-=\emptyset$ and $f'(c)=-m<0$.
Continuity of $f$ yields
$$\inf_{c\in\R} f(c) = \inf\{ f(c) : \gamma_{\sigma(1)}\leq c\leq \gamma_{\sigma(m)}\} 
	=   \min\{ f(c) : \gamma_{\sigma(1)}\leq c\leq \gamma_{\sigma(m)}\}.$$
\end{proof}
\begin{lemma}
For every $n\in S^1$ the value $f(c,n)$ is a local minimum of the function $f(\cdot,n)$ if and only if
$|J_-|+|J_0|- |J_+|\geq 0$ and $|J_+|+|J_0|- |J_-|\geq 0$.
A local minimum is strict if and only if $|J_-|+|J_0|- |J_+| > 0$ and $|J_+|+|J_0|- |J_-| > 0$.
\end{lemma}
\begin{proof}
Let $J_+$, $J_0$, $J_-$ be the decomposition of the index set for the line $\langle q,n\rangle = c$.
If $\max \{ \gamma_j-c : j\in J_-\} < \varepsilon < \min\{Ê\gamma_j-c: j\in J_+\} $, then the decomposition
of the index set for the line $\langle q,n\rangle = c+\varepsilon$ is equal to that for 
$\langle q,n\rangle = c$ and
\begin{align*}
f(c+\varepsilon,n) & = (|J_-|-|J_+|)(c+\varepsilon)  -  \sum_{j\in J_-}\langle p_j,n\rangle 
	+\sum_{j\in J_+} \langle p_j,n\rangle. +\sum_{j\in J_0} |\varepsilon|  \\
& = f(c,n) +\varepsilon \cdot  \begin{cases} |J_-|+|J_0|- |J_+| & \text{if } \varepsilon >0 \\
	|J_-|-|J_0|- |J_+| & \text{if } \varepsilon <0 \end{cases}.
\end{align*}
\end{proof}
\begin{cor}
\label{folg:geo-l1-Zerlegung-im-Optimum}
For every $n\in S^1$ the value $f(c,n)$ is a global minimum of the function $f(\cdot,n)$ if and only if
$| |J_+| - |J_-| | \leq |J_0|$.\\
Furthermore, $\min\{ f(c):c\in\R\} = f(c_0)$ if and only if $\gamma_k \leq c_0 \leq \gamma_l$ with
$$k=\left\lfloor \frac{m+1}{2}\right\rfloor,\, l = \left\lceil \frac{m+1}{2}\right\rceil.$$
\end{cor}
\begin{proof}
Replace $\beta$ by $\gamma$ in the proof of Corollary \ref{folg:alg-l1-Zerlegung-im-Optimum}.
\end{proof}
\begin{cor}
\label{folg:geo-l1-Verschiebung-in-Punkt}
Among all optimal lines with fixed normal vector $n$ 
there exists at least one containing one of the points $p_j$.
\end{cor}
\begin{proof}
If $b=\gamma_k$ and $b=\gamma_l$, then $J_0\ne \emptyset$.
\end{proof}
\begin{cor}
The function $f:\R\times S^1\to\R^{\geq 0}$ defined by $(c,n)\mapsto f(c,n)$ admits a global minimum. 
\end{cor}
\begin{proof}
Since the function $f$ is continuous and the set $S^1$ is compact, we obtain
$\inf\{ f(c,n) : c\in\R, n\in S^1\} = \min \{  f(\langle p_j,n\rangle,n) : n\in S^1, j=1,\ldots,m\}$.
\end{proof}
\subsubsection{Rotation}
For every point $p_k$ we restrict $f$ to lines containing $p_k$.
The normal vector of the lines is variable, but the condition $p_k\in \{ q\in \R^2 : c=\langle q,n\rangle\}$ 
implies $c=\langle p_k,n\rangle$.
We investigate the function $f:S^1\to\R$ given by
$$f(n):=f(\langle p_k,n\rangle, n) = \sum_{j=1}^{m} |\langle p_j,n\rangle-\langle p_k,n\rangle|
	= \sum_{j=1}^{m} |\langle p_j-p_k,n\rangle|= \sum_{j\ne k} |\langle \tilde{p}_j,n\rangle|,$$
where $\tilde{p}_j:=p_j-p_k \ne 0$ for all $j\ne k$.

As before, the change of coordinates $p_j\mapsto \tilde{p}_j$ corresponds to the translation of the origin to $p_k$.
For any normal vector $n\in S^1$ the decomposition of the index set with respect to the new coordinates 
is given by $J_+ = \{ \langle \tilde{p_j},n\rangle >0 \}$, $J_0 = \{ \langle \tilde{p_j},n\rangle = 0 \} \ni k$ 
and $J_- = \{ \langle \tilde{p_j},n\rangle <0 \}$ (see Figure \ref{fig:DrehGeo}). Now
$$f(n) = \sum_{j\in J_+} \langle \tilde{p}_j,n\rangle - \sum_{j\in J_-} \langle \tilde{p}_j,n\rangle
	=  \langle \tilde{p},n\rangle \text{ where } 
	\tilde{p} := \sum_{j\in J_+} \tilde{p}_j - \sum_{j\in J_-} \tilde{p}_j.$$
\begin{lemma}
\label{lem:geo-l1-Drehung-in-zwei-Punkte}
If $\langle \tilde{p}_j,n_0\rangle \ne 0$ for all $j\ne k$, then $\min\{ f(n):n\in S^1\} < f(n_0)$.
\end{lemma}
\begin{proof}
Note that $\tilde{p}_k=0$. Hence, $|J_0|\geq 1$.
If $J_+\cup J_- \ne \emptyset$, i.e., $\langle \tilde{p}_j,n\rangle \ne 0$ for at least one index $j\ne k$, 
then $f(n)=\langle \tilde{p},n\rangle >0$. Thus, $\tilde{p}\ne 0$.

Using the parametrization $n(t) = (\cos t, \sin t)^T$ of $S^1$ 
as in subsection \ref{subsubsec:geo-l2-Drehung} we define the function $h(t):=f(n(t))$.
If $\langle p_j,n\rangle \ne 0$ for all $j\ne k$, i.e., $J_0=\{Êk\}$, 
then $f$ is differentiable at the corresponding point.
Now
$$ \frac{d h}{d t} = \tilde{p}^T Jn = \langle \tilde{p},Jn\rangle, 
	\quad h''(t) = \tilde{p}^T JJn(t) = -\tilde{p}^T n(t) = \langle -\tilde{p},n(t)\rangle.$$ 
If $f(n(t))$ is a local extremum of $f$ and $J_0=\{ k\}$, then there exists $\lambda\ne 0$ such that 
$\tilde{p}=\lambda n$, since $h'(t)=0$ an $Jn \perp n$.
A local extremum $f(n(t)))$ of $f$ satisfying $J_0=\{ k\}$ is a local minimum if and only if $\lambda < 0$,
since $h''(t) = -\lambda\langle n,n\rangle = -\lambda$.
Set $\tilde{p}_+:=\sum_{j\in J_+}p_j$ and $\tilde{p}_-:=\sum_{j\in J_-}p_j$.
Note that $\tilde{p}_+ - \tilde{p}_- = \tilde{p}$, $\langle \tilde{p}_+,n\rangle >0$ and 
$\langle \tilde{p}_-,n\rangle <0$.
This implies $\langle \tilde{p}_+,\lambda n\rangle <0$, $\langle \tilde{p}_-,\lambda n\rangle >0$.
Thus, $\langle \tilde{p}_+ - \tilde{p}_-,\lambda n\rangle <0$
contradicting
$\langle \tilde{p}_+ - \tilde{p}_-,\lambda n\rangle = \langle \tilde{p},\tilde{p}\rangle =\| \tilde{p}\|^2 > 0$.
\end{proof}
\subsubsection{Summary of the geometric $L^1$-distance}
Corollary \ref{folg:geo-l1-Zerlegung-im-Optimum} and 
Lemma \ref{lem:geo-l1-Drehung-in-zwei-Punkte} imply the existence of a line containing 
two of the points $p_j$ among all lines with minimal geometric
$L^1$-distance to the set $\{ p_1,\ldots,p_m\}$.
These are at most $m(m-1)/2$ lines.
As in section \ref{sec:Totaler-algebraischer-Abstand}, it is sufficient to check the lines
which additionally satisfy $| |J_+| - |J_-| | \leq |J_0|$.

As before, we denote the line through $p_j$ and $p_k$ with $j > k$ by $g_{jk}$.
The normal vector of $g_{jk}$ is $n_{jk}:=J(p_j-p_k)\| p_j-p_k\|^{-1}$.
The line $g_{jk}$ is given by the equation $\langle q, n_{jk}\rangle = \langle p_k, n_{jk}\rangle =:c_{jk}$.
Let $E$ be the set of points in the domain  $f$ corresponding to optimal lines. More precisely,
\begin{equation}
E :=\left\{Ê\pm (c_{jk},n_{jk}) : f(c_{jk},n_{jk}) = \min\{ f(c,n) : c\in\R, n\in S^1\} \right\}.
\end{equation}
Since $f:\R\times S^1\to\R$ is only convex with respect to $c$, 
we perform the convex hull only in one direction and define
\begin{equation}
\bar{E} := \{ (c,n) \in \R\times S^1 : \exists  t\in[0,1], (c_0,n), (c_1,n)\in E : c=tc_0+(1-t)c_1\}.
\end{equation}
\begin{theorem}
\label{satz:geo-L1-Optimenge}
Let $p_j\in\R^2$, $j=1,\ldots,m$, be pairwise distinct points. \\
The line defined by the equation $\langle q,n\rangle = c$ with $n\inÊS^1$ and $c\in \R$ has
minimal geometric $L^1$-distance to the set $\{ p_1,\ldots,p_m\}$ if and only if  $(c,n)\in \bar{E}$.
\end{theorem}
\begin{proof}
The statement is a consequence of Corollaries \ref{folg:geo-l1-Zerlegung-im-Optimum} 
and \ref{folg:geo-l1-Verschiebung-in-Punkt}.
\end{proof}
\begin{rem}
As long as the set $E$ can be defined 
Theorem \ref{satz:geo-L1-Optimenge} remains true 
if the condition $p_j\ne p_k$ for all $j\ne k$ is omitted.
However, if $p_1=\ldots=p_m=\bar{p}$, then all lines containing $\bar{p}$ have zero distance to the
point set and are optimal. Note that only finitely many normal vectors occur in $E$ for a generic point set,
whereas optimal lines with any normal vector exist in the special case $p_1=\ldots=p_m=\bar{p}$.
\end{rem}
\subsubsection{Examples}
\label{subsubsec:Beispiele-geo-l1}
\paragraph{Three points:} Three pairwise distinct points $p_1$, $p_2$, $p_3$ form a triangle $D$. 
Let $A$ be the area of $D$. If $j>k$ and $l\ne j,k$, then $f(g_{jk}) = d(p_l,g_{jk})$.
Since $2A=\| p_j-p_k\| f(g_{jk})$ for all $j\ne k$, the line $g_{jk}$ is an optimal line if and only if 
$g_{jk}$ contains the longest edge of the triangle $D$.
If there are two longest edges ($\Z_2$-symmetry) or 
if $D$ is an equilateral triangle ($\mathfrak{S}_3$-symmetry),
then there exist two respectively three geometric $L^1$-optimal lines. 
\paragraph{Four points without symmetry admitting two optimal lines:}
\begin{wrapfigure}{r}{5,5cm}
  \mbox{\scalebox{.25}{\includegraphics{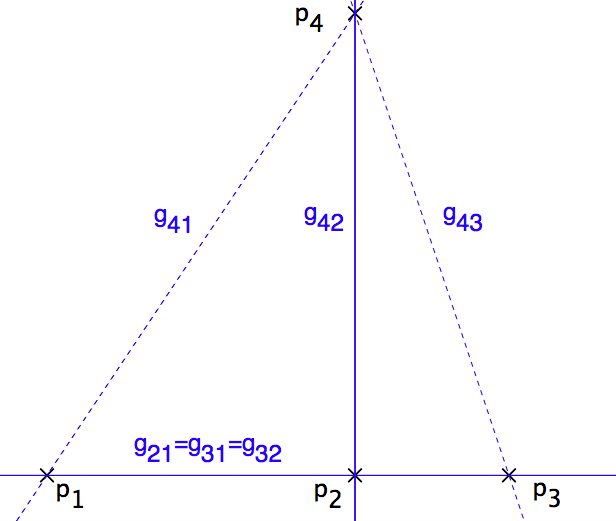}}}
\caption{}
\label{fig:geol1m4}
\end{wrapfigure}
Let us consider the four points $p_1=(0,0)^T$, $p_2=(2,0)^T$, $p_3=(3,0)^T$, $p_4=(2,3)^T$
(see Figure  \ref{fig:geol1m4}). 
If the lines $g_{43}$ or $g_{41}$ were optimal, then there would exist an optimal line containing 
exactly one of the points $p_j$. This would contradict Lemma \ref{lem:geo-l1-Drehung-in-zwei-Punkte}, 
since translation into the direction of $p_2$ does not change the value of $f$. 

The normal vectors of the lines $g_{21}=g_{31}=g_{32}$ and $g_{42}$ are $n_{21}=(0,1)^T$ 
respectively $n_{42}=(1,0)^T$. Now $f(0,n_{21}) = 3 = f(2,n_{42})$.
Since $n_{21}\ne \pm n_{42}$ there exist exactly two lines with minimal geometric $L^1$-distance
to the set $\{p_1,\ldots,p_4\}$.
\paragraph{Invariance under reflection for $m=5$:}
Consider the points $p_1=(-2,1)^T$, $p_2=(-1,-1)^T$, $p_3=(0,0)^T$, $p_4=(1,-1)^T$ and 
$p_5=(2,1)^T$ as in subsection \ref{subsubsec:Beispiele-alg-l1} and Figure \ref{fig:algl1m5}.
The lines $g_{21}$, $g_{42}$, $g_{54}$ and $g_{51}$ do not fulfill the condition
$|J_0| \geq | |J_+|- |J_-| |$, since for those lines $|J_0| = 2 < 3 = | |J_+|- |J_-| |$. 

Now $n_{53}= (1,-2)^T/\sqrt{5}$, $n_{43} = (1,1)^T/\sqrt{2}$, $n_{52}= (2,-3)^T/\sqrt{13}$, $
c_{53}=c_{43} = 0$ and $c_{52} = 1/\sqrt{13}$.
Using the reflection symmetry we obtain 
\begin{align*}
f(0,n_{31}) & = f(0,n_{53})= (4+1+3)/\sqrt{5} = 8/\sqrt{5}, \\
f(0,n_{32}) & = f(0,n_{43})= (1+2+3)/\sqrt{2} = 6/\sqrt{2} >  8\sqrt{5}, \\
f(c_{41},n_{41}) = f(c_{52},n_{52}) &= (8+1+4)\sqrt{13} = \sqrt{13} >  8\sqrt{5}.
\end{align*}
The lines $g_{31}$ and $g_{53}$ are not parallel.
Hence, $E=\bar{E}$.
There are exactly two lines with minimal geometric $L^1$-distance to the set $\{ p_1,\ldots,p_5\}$.
These are $g_{31}$ and $g_{53}$.
\section{Maximal distance - $L^\infty$-norm}
\label{sec:MaxAbstand}
In this section we minimize the largest of the distances of the single points $p_j$.
We consider the continuous function $f(g)= \max \{Ê  d_j: j=1,\ldots,m\} $.
The function $f$ is differentiable if there is exactly one largest $d_j$.
We show the existence of the global minimum of $f$ for the algebraic and the geometric distance.
The set of optimal lines is finite in both cases, because optimal lines are located in a special manner 
between the edges and vertices of the convex polytope generated by  $p_1,\ldots,p_m$.
Using the vertical distance the function $f$ becomes convex and the global minimum is then attained
at a unique line.
\subsection{Minimal algebraic $L^\infty$-distance}
\label{sec:Maximaler-algebraischer-Abstand}
Given points $p_j=(x_j,y_j)^T\in \R^2$, $j=1,\ldots,m$, we determine the linear function
$y(x) = ax+b$, $a,b\in\R$, with minimal algebraic $L^\infty$-distance to the set $\{Êp_1,\ldots,p_m\}$, 
i.e., the minimum of the continuous, piecewise linear and convex function $f:\R^2\to\R$ defined by
$f(a,b)=\max\{ | y_j-(ax_j+b)| : j=1,\ldots,m\}$.

As in the sections \ref{sec:LineareRegression} and \ref{sec:Totaler-algebraischer-Abstand} we
additionally assume that $x_j\ne x_k$ for all $j\ne k$ and decompose the index set $\{p_1,\ldots,p_m\}$
into $J_+=\{Êj:y_j > y(x_j)\}$, $J_0=\{Êj:y_j = y(x_j)\}$ and $J_-=\{Êj:y_j < y(x_j)\}$.
Now
$$f(a,b)  = \max \left\{ 0, \max\{ y_j-ax_j-b : j\in J_+\}, \max\{ ax_j+b-y_j : j\in J_-\}\right\}.$$
\subsubsection{Translation}
As in section \ref{subsubsec:Parallel-l1-alg}, we restrict $f$ to lines with arbitrary fixed slope $a\in\R$.
Again, set $\beta_j:=y_j-ax_j$ and let $\sigma$ be a permutation satisfying
$\beta_{\sigma(1)} \leq \ldots \leq \beta_{\sigma(m)}$ (see Figure \ref{fig:ZerlegAlg}). 
\begin{lemma}
\label{lem:alg-Linfty-opti-b}
$\min\{ f(a,b) :b\in\R\} = f(a,b_0)$ if and only if
$b_0=(\beta_{\sigma(m)}+\beta_{\sigma_{\sigma(1)}})/2$.
\end{lemma} 
\begin{proof}
The assertion follows from
$$f(a,b) = \max_{j=1,\ldots,m}\{ |\beta_j-b| \} = \max\{ |\beta_{\sigma(m)}-b|, |\beta_{\sigma(1)}-b| \}
	\geq (\beta_{\sigma(m)}-\beta_{\sigma(1)})/2$$
and $\max\{ |\beta_{\sigma(m)}-b|, |\beta_{\sigma(1)}-b| \} =  (\beta_{\sigma(m)}-\beta_{\sigma(1)})/2
	\Leftrightarrow b=(\beta_{\sigma(m)}+\beta_{\sigma_{\sigma(1)}})/2$.
\end{proof}
\begin{cor}
\label{folg:alg-linfinity-Verschiebung-in-Mittelpunkt}
If $\min\{ f(a,b):a,b\in\R\} = f(a_0,b_0)$, then there exist indices  $k\ne l$ such that
$f(a_0,b_0) = \beta_l-b_0 = b_0 - \beta_k$ and the point $(p_k+p_l)/2$ lies on the line defined by 
$y=a_0x+b_0$.
\end{cor}
\begin{proof}
The equation $f(a_0,b_0) = \beta_l-b_0 = b_0 - \beta_k$ holds for  $k=\sigma(1)$ and $l=\sigma(m)$.
It is easy to check that
$$y\left( \frac{x_k+x_l}{2}\right) = a_0\frac{x_k+x_l}{2}+b_0 
	= \frac{1}{2}(a_0x_k+a_0x_l+\beta_l+\beta_k) = \frac{y_k+y_l}{2}$$
for the function $y(x) = a_0x+b_0$.
\end{proof}
\subsubsection{Rotation}
\label{subsubsec:Drehung-linfty-alg} 
For any pair of indices $(k,l)$ with  $k \ne  l$ we consider the values of $f$ at lines containing
the point $(p_k+p_l)/2$.
The condition $y((x_k+x_l)/2) = (y_k+y_l)/2$ implies $2b=(y_k+y_l)-a(x_k+x_l)$.
We investigate the function$f:\R\to\R$ defined by
\begin{align*}
f(a) & := f(a,(y_k+y_l-a(x_k+x_l))/2) \\
& = \max\{ |y_j-ax_j-(y_k+y_l)/2+a(x_k+x_l)/2|: j=1,\ldots,m\} \\
& = \max\{ |\tilde{y}_j -a\tilde{x}_j| : j=1,\ldots,m\} 
\end{align*}
where $\tilde{x}_j := x_j-(x_k+x_l)/2$ and $\tilde{y}_j := y_j-(y_k+y_l)/2$.
Note that $\tilde{x}_k,\tilde{x}_lÊ\ne 0$, since the coordinates $x_j$ are pairwise distinct.
\begin{lemma}
\label{lem:alg-linfinity-Drehung-zu-Kantenparallele}
If $f(a_0) = \min\{ f(a) : a\inÊ\R\}$, $f(a_0) = \tilde{y}_l-a_0\tilde{x}_l = a_0\tilde{x}_k-\tilde{y}_k$ and
$f(a_0)>0$, then there exists an index $j\ne k,l$ such that  $f(a_0) = |\tilde{y}_j-a_0\tilde{x}_j|$.
\end{lemma}
\begin{proof}
Note that 
$|\tilde{y}_l-a\tilde{x}_l| = |\tilde{y}_k-a\tilde{x}_k|$ holds for all $a\in\R$, because
$\tilde{y}_l-a\tilde{x}_l = (x_l-x_k)/2-a(y_l-y_k)/2$ and $\tilde{y}_k-a\tilde{x}_k = (x_k-x_l)/2-a(y_k-y_l)/2$.
If the inequality $|\tilde{y}_j-a_0\tilde{x}_j| < |\tilde{y}_k-a_0\tilde{x}_k|$ holds for all $j\ne k,l$, 
then $f$ is differentiable at $a_0$ and $f'(a_0) = -\tilde{x}_l=\tilde{x}_k \ne 0$.
\end{proof}
\subsubsection{Existence and Uniqueness of the optimal line}
\label{subsubsec:alg-linfinity-Eindeutigkeit}
\begin{cor}
The function $f:\R^2\to\R$ defined by $(a,b)\mapsto \max\{ |y_j-ax_j-b|:j=1\ldots, m\}$ admits a global
minimum.
\end{cor}
\begin{proof}
For any $k\ne l$ the continuous functions investigated in subsection \ref{subsubsec:Drehung-linfty-alg}  admit a global minimum, since $\lim_{a\to\pm\infty} |\tilde{y}_j-a\tilde{x}_j| = \infty$ 
for all $j$ with $\tilde{x}_j\ne 0$.
\end{proof}
If $0 < f(a_0,b_0)= \min \{ f(a,b):a,b\in\R\}$, 
then Corollary \ref{folg:alg-linfinity-Verschiebung-in-Mittelpunkt} and 
Lemma \ref{lem:alg-linfinity-Drehung-zu-Kantenparallele} imply the existence of three pairwise distinct 
indices $k_1,k_2,k_3$ with the following properties:
The line $g_{k_1k_2}$ through $p_{k_1}$ and $p_{k_2}$ is parallel to the line given by $y=a_0x+b_0$, 
$f(a_0,b_0) = |y_{k_i}-a_0x_{k_i}-b_0|$ for $i=1,2,3$, and $p_{k_3}\not\in g_{k_1k_2}$.
Thus, the value $f(a_0,b_0)$ is half of the vertical distance between the point $p_{k_3}$ and the
line $g_{k_1k_2}$, i.e.,
$$f(a_0,b_0) = \frac{1}{2}
\left\vert y_{k_3}-y_{k_1} - \frac{y_{k_2}-y_{k_1}}{x_{k_2}-x_{k_1}}(x_{k_3} -x_{k_1})\right\vert.$$
\begin{lemma}
If $f(a_0,b_0) = f(a_1,b_1) = \min\{ f(a,b) :a,b\in\R\}$, then $a_0=a_1$ and $b_0=b_1$.
\end{lemma}
\begin{proof}
There are only finitely many triples $(k_1,k_2,k_3)$ with the properties described in 
Lemma \ref{lem:alg-linfinity-Drehung-zu-Kantenparallele}. 
Therefore, the global minimum of the function $f(a,b)$ is attained only at finitely many lines. 
Since $f$ is a convex function, the set of optimal lines is a convex set.
Finite convex sets contain at most one element.

This lemma can be proven without making use of the general properties of convex functions directly.
We provide this second proof to clarify the link between the uniqueness of the optimal line 
and the standard assumption $x_j\ne x_k$ for all $j\ne k$:
The condition $d=f(a_0,b_0)=f(a_1,b_1)$ implies
$$p_jÊ\in \{ (x,y)^T : d\geq |y-a_ix-b_i|,  i = 0, 1 \} =:M \text{ for all } j.$$

If $a_0=a_1$, then all points  $p_j$ are contained in the intersection of two parallel stripes of vertical width $2d$. This intersection is a parallel stripe of width $< 2d$ if and only if $b_0\ne b_1$, since
$$M=\left\{ (x,y) : d+\min(b_0,b_1) \geq y-a_0x \geq -d +\max(b_1,b_0)\right\}$$
for $a_0=a_1$.

If $a_1 \ne a_0$, then $M$ is a parallelogram.
Using the triangle inequality we obtain 
$M \subset \left\{  (x,y) : 2d\geq \left\vert 2y-(a_1+a_0)x-(b_1+b_0) \right\vert \right\}$.
Since $d$ is the global minimum of the function $f$, it follows that
$f(\frac{a_0+a_1}{2},\frac{b_0+b_1}{2})  = d$.
A small calculation, that again uses the triangle inequality, shows that
$$M \cap \left\{  (x,y) : 2d =  \left\vert 2y-(a_1+a_0)x-(b_1+b_0) \right\vert \right\} 
	= \{ (x_0,a_0x_0+b_0\pm d)\}$$
for $x_0=-(b_1-b_0)/(a_1-a_0)$.
This fact contradicts Corollary \ref{folg:alg-linfinity-Verschiebung-in-Mittelpunkt},
since the coordinates $x_j$ are pairwise distinct.
\end{proof}
\subsubsection{Summary of the algebraic $L^\infty$-distance}
\label{subsubsec:Zusammenfassung-max-alg}
The notions convex hull of the set $\{p_1,\ldots,p_m\}$, edge and vertex 
are appropriate to effectively characterize the triples of indices with the properties 
being specified in subsection \ref{subsubsec:alg-linfinity-Eindeutigkeit} .

The convex hull $H(\{p_1,\ldots,p_m\})$ of the points $p_1,\ldots,p_m$ is denoted by $P$.
A point $p_k$ is a vertex of the polytope $P$, if $p_k$ is not contained in the convex hull of the
remaining points $p_j$, $j\ne k$. 
The set of vertices is denoted by $V$, i.e., $V := \{Êp_k : p_k \not\in H(\{ p_j:j\ne k\}) \}$.
The boundary $\partial{P}$ of the set $P$ consists of line segments of the form
$S_{kl} :=\{ p_l+t(p_k-p_l): 0\leq t\leq 1\}$ with $k,l\in V$.
\begin{theorem}
\label{satz:alg-Linfty-Optimenge}
Let $p_j=(x_j,y_j)^T$, $j=1,\ldots,m$, such that  $x_j\ne x_k$ for all $j\ne k$.\\
There exists a unique linear function with minimal algebraic $L^\infty$-distance to the set
$\{p_1,\ldots,p_m\}$.
Moreover,
\begin{align*}
\min_{a,b\in\R} f(a,b) & = \min_{a,b\in\R} \max_{j=1,\ldots,m} |y_j-ax_j-b|
	= \min_{a,b\in\R} \max_{p_j\in V} |y_j-ax_j-b| \\
& = \frac{1}{2} \min_{S_{kl}\in \partial{P}} \max_{p_j\in V}
	\left\vert y_{j}-y_{k} - \frac{y_{k}-y_{l}}{x_{k}-x_{l}}(x_{j} -x_{k})\right\vert.
\end{align*}
If $0 < f(a_0,b_0)=\min\{ f(a,b):a,b\in\R\}$, then there exist pairwise distinct indices $k_1,k_2,k_3 \in V$
such that $f(a_0,b_0) = |y_{k_i}-a_0x_{k_i}-b_0|$ for $i=1,2,3$ and $S_{k_1k_2}\subset \partial P$.
\end{theorem}
\begin{cor}
\label{folg:algLinfSpiegelsymmetrie}
If the set $\{ p_1,\ldots,p_m\}$ is invariant under the reflection in the line defined by $x=\bar{x}$, 
then the linear function with minimal algebraic $L^\infty$-distance is
$$y(x) \equiv \frac{1}{2}\left(\max\{ y_j:j=1,\ldots,m\} +\min\{ y_j:j=1,\ldots,m\} \right).$$ 
\end{cor}
\begin{proof}
Similar to the proof of Lemma \ref{folg:algL2Spiegelsymmetrie} this assertion follows from the
uniqueness of the optimal line and Lemma \ref{lem:alg-Linfty-opti-b}.
\end{proof}
\begin{rem}
Note that the optimal lines with minimal algebraic $L^\infty$-distance depend only on the convex hull
$P$ of the points $p_1,\ldots,p_m$. In particular, the minimum of $f$ is equal to zero if and only if 
$P$ is a line segment, i.e., there are at most two vertices.
As long as only edges not parallel to the $y$-axis are considered, 
Theorem \ref{satz:alg-Linfty-Optimenge} remains true, if the condition $x_j\ne x_k$ for all $j\ne k$ 
is weakened to the existence of indices $j\ne k$ satisfying $x_j\ne x_k$.

However, if $x_1=\ldots=x_m=\bar{x}$, then $P$ is a line segment parallel to the $y$-axis.
The linear functionen $y(x)=a(x-\bar{x})+b$ has minimal algebraic $L^\infty$-distance to the
set $\{p_1,\ldots,p_m\}$ if and only if  $a\in\R$ and 
$$b=\frac{1}{2}\left(\min\{ y_1,\ldots,y_m\}+\max\{ y_1,\ldots,y_m\}\right).$$
In contrast to Theorem \ref{satz:alg-Linfty-Optimenge}, these are infinitely many optimal lines.
\end{rem}
\subsubsection{Examples}
\label{subsubsec:Beispiele-alg-linfty}
\paragraph{Three points:}
As in the first example in subsection \ref{subsubsec:Beispiele-alg-l1} we consider three points
$p_j=(x_j,y_j)^T\in\R^2$ with  $x_1<x_2<x_3$. 
If $p_1,p_2,p_3$ are not collinear, then $E=\{ p_1,p_2,p_3\}$. 
Applying the results of subsection \ref{subsubsec:Beispiele-alg-l1} we conclude that
the line parallel to $S_{31}$ with equal vertical distance to $S_{31}$ and $p_2$ 
has minimal algebraic $L^\infty$-distance to $\{p_1,p_2,p_3\}$, 
because $p_2$ is the given point with smallest vertical distance 
to the corresponding opposite side of the triangle.
\paragraph{Four points:} 
We consider the examples concerning four given points 
of the subsections \ref{subsubsec:Beispiele-alg-l1} (see Figure \ref{fig:algl1m4}) and
\ref{subsubsec:Beispiele-geo-l1} (see Figure \ref{fig:geol1m4}).
The linear functions with minimal algebraic $L^\infty$-distance can be quickly determined 
for these given point sets, because three of the four given points are collinear in both cases.
Hence, the set of vertices consists of three elements.
As in the paragraph above, the optimal line is parallel to the line through the two vertices with
smallest respectively largest $x$-coordinate
(see Figures \ref{fig:alglinfm4} and \ref{fig:alglinfm4p}).
\begin{figure}
\noindent
\begin{minipage}[t]{.49\linewidth}
  \mbox{\scalebox{.26}{\includegraphics{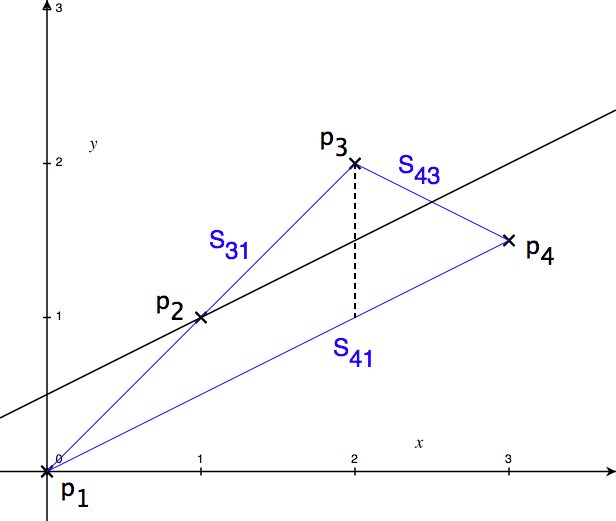}}}
\caption{}
\label{fig:alglinfm4}
\end{minipage}
\begin{minipage}[t]{.50\linewidth}
  \mbox{\scalebox{.26}{\includegraphics{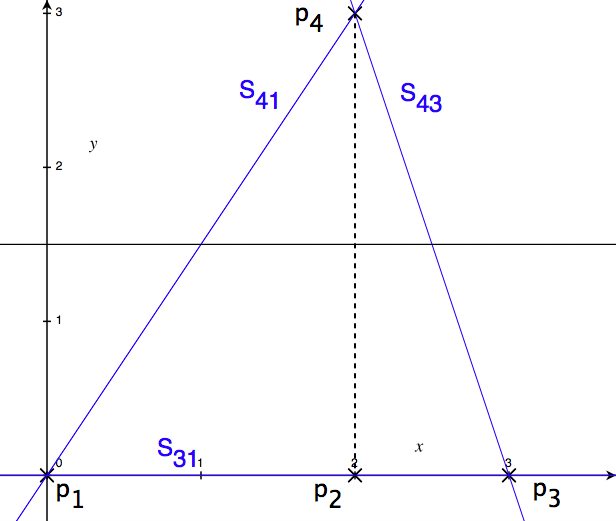}}}
\caption{}
\label{fig:alglinfm4p}
\end{minipage}
\end{figure}
\paragraph{Invariance under reflection symmetry:}
The set of points $p_1,\ldots,p_5$ in Figure \ref{fig:algl1m5} ist invariant under reflection in the $y$-axis.
Corollary \ref{folg:algLinfSpiegelsymmetrie} leads to the optimal linear function $y(x) \equiv 0$,
since $\bar{p} = (0,0)^T$. 
\subsection{Minimal geometric $L^\infty$-distance}
Given pairwise distinct points $p_j\in \R^2$, $j=1,\ldots,m$, we determine lines given by
$g = \{ q\in \R^2 : c=\langle q,n\rangle\}$, $c\in\R$, $n\in S^1$ with minimal geometric 
$L^\infty$-distance to the set $\{Êp_1,\ldots,p_m\}$, i.e., the minimum of the  function
$f:\R\times S^1\to\R$ defined by $f(c,n)=\max\{| c-\langle p_j,n\rangle| : j=1,\ldots,m\}$.

As in subsection \ref{subsec:geo-l1} we decompose the index set $\{ 1,\ldots,m\}$.
For every line $\{Êq\in\R^2: \langle q,n\rangle = c\}$ set
$J_+:=\{j: \langle p_j,n\rangle > c\}$, $J_0:=\{j: \langle p_j,n\rangle = c\}$ and
$J_-:=\{j: \langle p_j,n\rangle < c\}$ (see Figure \ref{fig:ZerlegGeo}).
Now
$$f(c,n) = \max\{ 0, \max\{ \langle p_j,n\rangle -c : j\in J_+\},  \max\{ c-\langle p_j,n\rangle: j\in J_-\}\} .$$
Similar to the investigation of the algebraic $L^\infty$-distance in section
\ref{sec:Maximaler-algebraischer-Abstand} we show that the function $f$ admits a global minimum
and describe optimal lines by means of edges and vertices of the convex hull of $p_1,\ldots,p_m$. 
Transferring the statements of section \ref{sec:Maximaler-algebraischer-Abstand},
note that $f(c,n)$ is convex only in the variable $c$.
\subsubsection{Translation}
Set $\gamma_j:=\langle p_j,n\rangle$. Let $\sigma$ be a permutation satisfying
$ \gamma_{\sigma(1)}\leq \ldots \leq \gamma_{\sigma(m)}$ 
as in subsection \ref{subsubsec:Parallel-l1-geo} (see Figure \ref{fig:ZerlegGeo}).
\begin{lemma}
It holds $\min\{ f(c,n) :c\in\R\} = f(c_0,n)$ if and only if 
$c_0=(\gamma_{\sigma(m)}+\gamma_{\sigma_{\sigma(1)}})/2$.
\end{lemma} 
\begin{proof}
Replace $b$ by $c$, $\beta$ by $\gamma$ and $a$ by $n$ 
in the proof of Lemma \ref{lem:alg-Linfty-opti-b}.
\end{proof}
\begin{cor}
\label{folg:geo-linfinity-Verschiebung-in-Mittelpunkt}
If $\min\{ f(c,n):c\in\R, n\in S^1\} = f(c_0,n_0)$, then there exist indices $k\ne l$ satisfying
$f(c_0,n_0) = \gamma_l-c_0 = c_0 - \gamma_k$ and $\langle (p_k+p_l)/2,n_0\rangle = c_0$.
\end{cor}
\begin{proof}
Replace  $\beta$ by $\gamma$ and  $b_0$ by $c_0$ 
in the proof of Lemma \ref{folg:alg-linfinity-Verschiebung-in-Mittelpunkt}.
\end{proof}
\begin{cor}
The function $f:\R\times S^1\to\R$ defined by the assignment
$(c,n) \mapsto \max\{| c-\langle p_j,n\rangle| : j=1,\ldots,m\}$ admits a global minimum.
\end{cor}
\begin{proof}
There are only finitely many points of the form $(p_k+p_l)/2$ with $k>l$, $S^1$ is a compact set and
the function $f$ is continuous.
\end{proof}
\subsubsection{Rotation}
For any pair of indices $(k,l)$ with $k \ne  l$ we consider the values of $f$ at lines containing the point 
$(p_k+p_l)/2$.
The condition $(p_k+p_l)/2 \in g = \{Êq : \langle q,n\rangle = c \}$ implies 
$\langle (p_k+p_l)/2,n\rangle = c$. 
Therefore, we investigate the function $f:S^1 \to\R$ defined by
\begin{align*}
f(n) & := f(\langle (p_k+p_l)/2,n\rangle,n) \\
& = \max\{| \langle (p_k+p_l)/2,n\rangle-\langle p_j,n\rangle| : j=1,\ldots,m\} \\
& = \max\{| \langle p_j-(p_k+p_l)/2,n\rangle| : j=1,\ldots,m\} \\
& = \max\{| \langle \tilde{p}_j,n\rangle| : j=1,\ldots,m\} \text{ where } \tilde{p}_j := p_j-(p_k+p_l)/2.
\end{align*}
Note that $\tilde{p}_k = - \tilde{p}_l \ne 0$, since $\tilde{p}_k = (p_k-p_l)/2$, $\tilde{p}_l = (p_l-p_k)/2$ 
and $p_k\ne p_l$.
\begin{figure}
\noindent
\begin{minipage}[t]{.50\linewidth}
 \mbox{\scalebox{.27}{\includegraphics{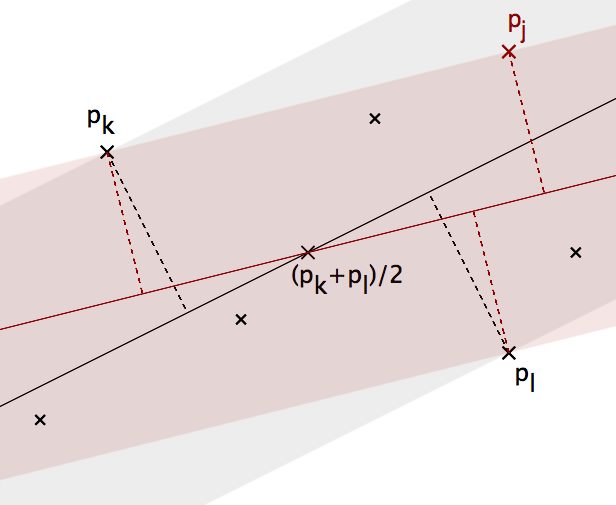}}}
\caption{}
\label{fig:LinfDreh}
\end{minipage}
\begin{minipage}[t]{.49\linewidth}
  \mbox{\scalebox{.26}{\includegraphics{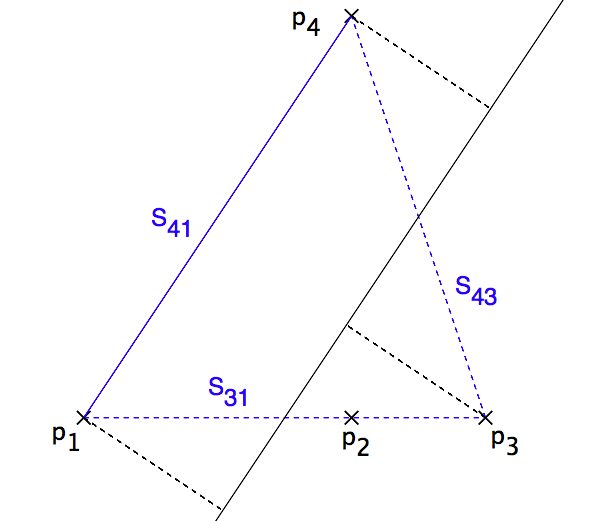}}}
\caption{}
\label{fig:geolinfm4}
\end{minipage}
\end{figure}
\begin{lemma}
\label{lem:geo-linfinity-Drehung-zu-Kantenparallele}
If $f(n_0) = \min\{ f(n) : n\inÊS^1\}$,   
$f(n_0) = \langle \tilde{p}_l,n_0\rangle = - \langle \tilde{p}_k,n_0\rangle$ and $f(n_0)>0$,
then there exists an index $j\ne k,l$ such that $f(n_0) = |\langle \tilde{p}_j,n_0\rangle|$.
\end{lemma}
\begin{proof}
Note that $\langle \tilde{p}_l,n\rangle = -\langle \tilde{p}_k,n\rangle$ for all $n\in S^1$, 
since $\tilde{p}_k= -\tilde{p}_l$.
If the inequality $|\langle \tilde{p}_j,n(t_0)\rangle|< f(n(t_0))$ holds for all $j\ne k,l$, 
then the function $h(t):=f(n(t))$ given by the parametrization  $n(t) = (\cos t,\sin t)^T$ 
is differentiable at $t_0$.
Note that $l\in J_+$. Now, $h'(t_0) = \tilde{p}_l^T J n(t_0) = 0$ if and only if 
there exists $\lambda >0$ such that $\tilde{p}_l = \lambda n$.
If $h'(t_0) = 0$, then $h''(t_0) = \tilde{p}_l^T JJn(t_0) = -\lambda<0$, since $n\in S^1$ and $J^2 = -\Id$.
Consequently, if $|\langle \tilde{p}_j,n(t)\rangle|< h(t)$ for all $j\ne k,l$, 
then $h(t)$ is not a local minimum. (see Figure \ref{fig:LinfDreh})
\end{proof}
\subsubsection{Summary of the geometric $L^\infty$-distance}
Using the notation of subsection \ref{subsubsec:Zusammenfassung-max-alg} for the convex hull 
and its edges and vertices we characterize lines with minimal geometric $L^\infty$-distance.
\begin{theorem}
\label{satz:geo-Linfty-Optimenge}
Let $p_j\in\R^2$, $j=1,\ldots,m$, be pairwise distinct points.\\
There exists a line with minimal geometric $L^\infty$-distance to the set $\{p_1,\ldots,p_m \}$. 
If $0<f(c_0,n_0)=\min\{ f(c,n):c\in\R, n\in S^1\}$, 
then there exists pairwise distinct indices $k_1,k_2,k_3 \in V$ such that 
$f(c_0,n_0) = |\langle p_{k_i},n_0\rangle - c_0|$ for $i=1,2,3$ and $S_{k_1k_2}\subset \partial P$.
Moreover,
\begin{align*}
\min_{c\in\R, n\in S^1} f(c,n) & = \min_{c\in\R, n\in S^1} \max_{j=1,\ldots,m} |\langle p_j,n\rangle -c|
	= \min_{c\in\R} \max_{p_j\in V} |\langle p_j,n\rangle -c|\\
& = \frac{1}{2} \min_{S_{kl}\in \partial{P}} \max_{p_j\in V}
	\frac{|\langle p_j-p_k,J(p_l-p_k)\rangle |}{\| p_l-p_k\| }.
\end{align*}
In particular, the set of optimal lines is finite.
\end{theorem}
\begin{proof}
If $p_1,\ldots,p_m$ are not collinear, 
then Lemma \ref{lem:geo-linfinity-Drehung-zu-Kantenparallele} and
Corollary \ref{folg:geo-linfinity-Verschiebung-in-Mittelpunkt} imply the following properties of a line $g$
with minimal geometric $L^\infty$-distance to the set  $\{p_1,\ldots,p_m \}$:
The line $g$ is parallel to an edge $S_{k_1k_2}$ of the polytope $P$. 
The geometric distance between any vertex in $V$ and $g$ is less or equal to the geometric distance
between $S_{k_1k_2}$ and $g$.
There exists a vertex $p_{k_3}\in E$ such that $p_{k_3}\not\in S_{k_1k_2}$ and 
the geometric distance between $g$ and $p_{k_3}$ is equal to that between $g$ and $S_{k_1k_2}$.
Thus, the minimum of $f$ is half of the geometric distance between $S_{k_1k_2}$ and $p_{k_3}$.

Since $n=J(p_l-p_k) \|p_l-p_k\|^{-1}$ is a normal vector of the line through the points $p_k\ne p_l$,
the geometric distance between $p_j$ and $S_{kl}$ is given by 
$|\langle p_j-p_k,J(p_l-p_k)\rangle | \| p_l-p_k\|^{-1}$. 
\end{proof}
\begin{rem}
The geometric, just as the algebraic, $L^\infty$-distance to a point set $\{p_1,\ldots,p_m\}$ depends
only on the convex hull of the points $p_1,\ldots,p_m$.
Hence, Theorem \ref{satz:geo-Linfty-Optimenge} remains true
if the condition $p_j\ne p_k$ for all $j\ne k$ is weakened to the existence two indices $j\ne k$ such that
$p_j\ne p_k$.
If $p_1=\ldots=p_m=\bar{p}$, then $P= V =\{\bar{p}\}$ and any line through $\bar{p}$ has minimal distance zero.
\end{rem}
\subsubsection{Examples}
\begin{figure}
\noindent
\begin{minipage}[t]{.50\linewidth}
  \mbox{\scalebox{.27}{\includegraphics{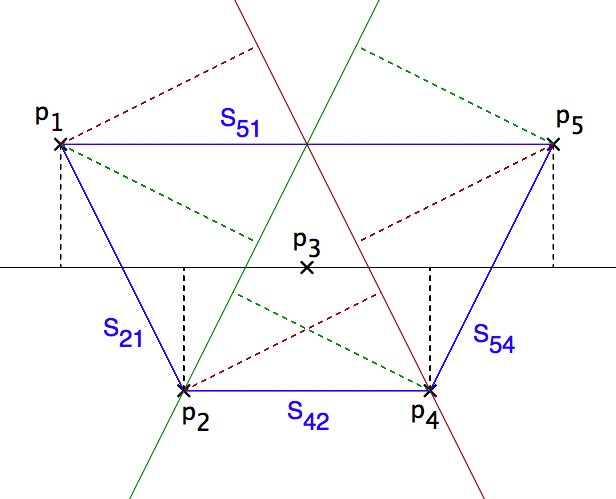}}}
\caption{}
\label{fig:geolinfm5}
\end{minipage}
\begin{minipage}[t]{.49\linewidth}
  \mbox{\scalebox{.27}{\includegraphics{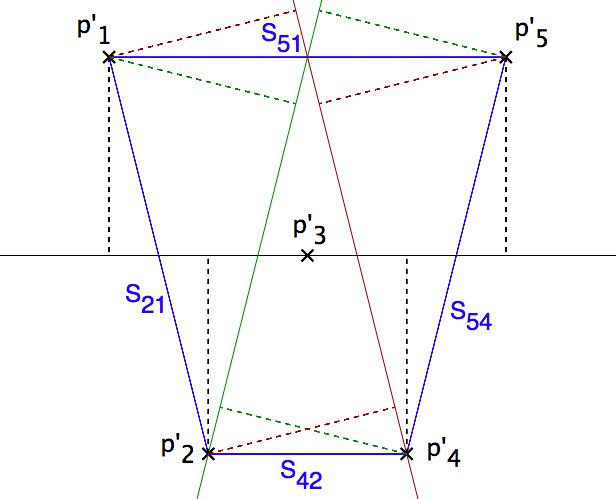}}}
\caption{}
\label{fig:geolinfm5s}
\end{minipage}
\end{figure}
\paragraph{Three points:}
We consider three pairwise distinct points $p_1, p_2, p_3 \in \R^2$ 
as in the first example of subsection \ref{subsubsec:Beispiele-geo-l1}. 
If $p_1,p_2,p_3$ are not collinear, then $V=\{ p_1,p_2,p_3\}$. 
A line $g$ has minimal geometric $L^\infty$-distance to $\{ p_1,p_2,p_3\}$ if and only if
$g$ is parallel to the longest side $s$ of the triangle 
and the geometric distance between $s$ and $g$ is equal to 
the geometric distance between $g$ and the point opposite to $s$.
\paragraph{Four points:}
Let us consider the four points $p_1=(0,0)^T$, $p_2=(2,0)^T$, $p_3=(3,0)^T$ and $p_4=(2,3)^T$ 
(see Figure \ref{fig:geolinfm4}). Obviously, $V=\{p_1,p_3,p_4\}$. 
The edge $S_{41}$ is the unique longest side of the triangle $p_1p_3p_4$. 
There exists a unique line with minimal geometric $L^\infty$-distance to $\{p_1,\ldots,p_4\}$.
This optimal line is parallel to $S_{41}$ and has the same geometric distance to $S_{41}$ and to $p_3$. 
Compare this line to the lines with minimal algebraic $L^\infty$-distance (Figure \ref{fig:alglinfm4p})
and with minimal geometric $L^1$-distance (Figure \ref{fig:geol1m4}),
they are all different.
\paragraph{Invariance under reflection symmetry:}
Let us consider the five points $p_1=(-2,1)^T$, $p_2=(-1,-1)^T$, $p_3=(0,0)^T$, $p_4=(1,-1)^T$ and 
$p_5=(2,1)^T$ as in the subsections \ref{subsubsec:Beispiele-alg-l1}, \ref{subsubsec:Beispiele-geo-l1} and \ref{subsubsec:Beispiele-alg-linfty} and Figure \ref{fig:algl1m5}.
The convex hull of the five points is a trapezium with $V=\{ p_1,p_2,p_4,p_5\}$ 
(see Figure \ref{fig:geolinfm5}). 
We denote the geometric distance between the vertex $p_j$ and the edge $S_{kl}$ by $d(S_{kl},p_j)$. 
The edges $S_{51}$ and $S_{42}$ are parallel.
It is easy to see that $d(S_{42},p_5) = d(S_{42},p_1) = d(S_{51},p_2) = d(S_{51},p_4) = 2$.
The set $\{p_1,\ldots,p_5\}$ is invariant under reflection in the $y$-axis.
Hence, $d(S_{21},p_5) = d(S_{54},p_1) = 2 d(S_{21},p_4) = 2d(S_{54},p_2)$.
Using the formula given in Theorem \ref{satz:geo-Linfty-Optimenge} we calculate
$$d(S_{21},p_5) = \frac{|\langle p_5-p_2,J(p_1-p_2) \rangle |}{ \| p_1-p_2\|}
	= \frac{|(3,2)J(-1,2)^T |}{ \sqrt{5}} = \frac{8}{\sqrt{5}} >2 .$$
Thus, the unique line with minimal geometric $L^\infty$-distance to the point set $\{ p_1,\ldots,p_5\}$ is
given by the equation $y=0$.

In this example the line with minimal geometric $L^\infty$-distance is unique and coincides with the
line with minimal algebraic $L^\infty$-distance.
In contrast to Corollary \ref{folg:algLinfSpiegelsymmetrie},
the invariance under reflection in the line $x=\bar{x}$
does not imply the existence of an optimal line of the form $y=\bar{y}$ 
for the geometric $L^\infty$-distance, since the optimal line is not unique.
Scaling the $y$-coordinate makes this different behavior more obvious.

We consider the linear map $A$ given by $(x,y)\mapsto (x,\lambda y)$ 
with $0<\lambda \in\R$ and $p'_j:=A(p_j)$.
Now $p'_1=(-2,\lambda)^T$, $p'_2=(-1,-\lambda)^T$, $p'_3=(0,0)^T$, $p'_4=(1,-\lambda)^T$ 
and $p'_5=(2,\lambda)^T$.
The set $\{p'_1,\ldots,p'_5\}$ is invariant under reflection in the $y$-axis  for all $\lambda\in\R$. 
Hence, the linear function $y(x)\equiv 0$ has minimal algebraic $L^\infty$-distance to the set
$\{p'_1,\ldots,p'_5\}$, since $\bar{p}' = (0,0)$ for all $\lambda\in\R$.
Similar to the calculations above we obtain $V=\{p'_1,p'_2,p'_4,p'_5\}$,  
$d(S_{42},p'_1) = d(S_{42},p'_5) = d(S_{51},p'_2) = d(S_{51},p'_4) = 2\lambda$,
$d(S_{21},p'_5) = d(S_{54},p'_1) = 2 d(S_{21},p'_4) = 2d(S_{54},p'_2)$ and
$$d(S_{21},p'_5) = \frac{|\langle p'_5-p'_2,J(p'_1-p'_2) \rangle |}{ \| p'_1-p'_2\|}
= \frac{|(3,2\lambda)J(-1,2\lambda)^T |}{ \sqrt{1+4\lambda^2}} = \frac{8\lambda}{\sqrt{1+4\lambda^2}}.$$
Figure \ref{fig:geolinfm5s} shows the situation for the case $\lambda=2$.
For
$$d(S_{42},p'_1) < d(S_{21},p'_5) \Leftrightarrow \sqrt{1+4\lambda^2}<4 \Leftrightarrow 
	|\lambda| < \frac{\sqrt{15}}{2},$$
the set of lines with minimal geometric $L^\infty$-distance to the set $\{p'_1,\ldots,p'_5\}$
depends on $\lambda>0$ as follows:
For $\lambda>\sqrt{15}/2$ there are exactly two optimal lines, these are $y=\lambda\pm 2\lambda x$.
For $0 < \lambda<\sqrt{15}/2$ the line  $y=0$ is the unique optimal line.
For $\lambda = \sqrt{15}/2$ there are exactly three optimal lines, 
these are $y=0$ and $y=\lambda\pm 2\lambda x$. 
\section{$L^p$-Norm}
In this section we want to minimize the $L^p$-norm of the vector $(d_1,\ldots,d_m)$.
We consider the function $f(g) = \sum_{j=1}^m d_j^p$ for $p>1$.
This function is continuously differentiable, 
since the function $\R\to\R$ defined by $x\mapsto |x|^p$ has this property for $p>1$.
The case $p=2$ has been discussed in section \ref{sec:L2}.

We show that the function $f$ admits a global minimum for every $p>1$.
For $p\ne 2$ the critical points of $f$ are the solutions of a system of nonlinear equations,
which, in general,  can not be explicitly resolved by a closed formula.

The function $f(g)$ is convex in some of the parameters of the line $g$.
We apply the following lemmas to prove the existence of a global minimum and  
properties of the set of optimal lines in the succeeding subsections.
Except for a few special examples, e.g., $m=3$ or symmetries of the point set, 
it is impossible to provide explicit formulas of $L^p$-optimal lines.
Therefore, the global minimum of $f$ can be determined in most of the situations only numerically.
\begin{lemma}
\label{lem:Lp-Existenz-globMin}
Let $\xi_j\in\R^n$, $\eta_j\in\R$ and $p\geq 1$. \\
If $\cap_{j=1}^m \xi_j^\perp = \cap_{j=1}^m \{ x \in \R^n: \xi_j^Tx=0\} = \{ 0\}$, 
then the function $f:\R^n\to\R$ defined by $f(x) = \sum_{j=1}^m |\xi_j^T x +\eta_j|^p$ 
admits a global minimum.
\end{lemma}
\begin{proof}
The function $f$ is continuous. Set $\tilde{x}:=(x,1)$ and $\tilde{\xi}_j:=(\xi_j,\eta_j)$.
Now $f(x) = \sum_{j=1}^m |\tilde{\xi}_j^T\tilde{x}|^p$ and
$$f(x) = \| \tilde{x}\|^p  \sum_{j=1}^m \left|\tilde{\xi}_j^T\frac{\tilde{x}}{\|\tilde{x}\|}\right|^p
	\text{ for } \tilde{x}\ne 0.$$
The function $\tilde{x}\mapsto  \sum_{j=1}^m |\tilde{\xi}_j^T\tilde{x}|^p$ is continuous and 
admits a global minimum $m\geq 0$ on the compact set $\{ \|\tilde{x}\| = 1 \} \subset \R^{n+1}$.

If $m=0$, then there exist $\tilde{x}_0=(x_0,\lambda)$ with $x_0\in\R^n$ and $\lambda\in\R$ such that
$0= \tilde{\xi}_j^T\tilde{x}_0 = \xi^T_jx_0+\eta_j\lambda$ for all $j$ and $\|x_0^2\|^2+\lambda^2=1$.
The assumption $\lambda=0$ would imply $x_0\ne 0$ and $\xi_j^Tx_0=0$ for all $j$ contradicting 
$\cap_{j=1}^m \xi_j^\perp = \{ 0\}$. 
Hence, $\lambda\ne 0$ and $0=f(x_0/\lambda)$ is the global minimum of the function $f$.

If  $m>0$, then
$f(x) \geq (\|x\|^2+1)^{p/2} m \geq f(0) \geq 0$ for all $x$ satisfying $\|x\|^2 \geq (f(0)/m)^{2/p} -1=:R$.
Continuity of $f$ and compactness of the set $\{ \| x\|^2\leq R\} \subset \R^n$ 
yield the existence of a global minimum.
\end{proof}
\begin{lemma}
\label{lem:Lp-Eindeutigkeit-globMin}
Let $f:\R^n\to\R$ be defined by $f(x) =  \sum_{j=1}^m |\xi_j^T x +\eta_j|^p$ with
$\eta_j\in\R$, $p> 1$ and $\xi_j\in\R^n$ satisfying 
$\cap_{j=1}^m \xi_j^\perp = \cap_{j=1}^m \{ v \in \R^n: \xi_j^Tv=0\} = \{ 0\}$.\\
If $x_0,x_1\in\R^n$ such that $f(x_0)=f(x_1)=\min\{ f(x):x\in\R^n\}$, then $x_0=x_1$.
\end{lemma}
\begin{proof}
For $f$ is convex, $f(x_t)=f(x_0)=f(x_1)$ for all $x_t=tx_1+(1-t)x_0$ with $0\leq t\leq 1$.
We define $I:=\{ j: \xi_j^Tx_0+\eta_j=\xi_j^Tx_1+\eta_j=0\} \subset \{ 1,\ldots,m\}$ and 
$K:=\cap_{j\in I}\{ x:\xi_j^Tx+\eta_j=0 \} \subset \R^n$. 
For $K$ is the set of solutions of a system of linear equations, $K$ is an affine subspace of $\R^n$.
Thus, $K$ is  a convex set. In particular, $x_t\in K$ for all $0\leq t\leq 1$.

The restriction of $f$ to $K$ is given by $f|_K(x) = \sum_{j\not\in I} |\xi_j^Tx+\eta_j|^p$.
The functions $f$ and $f|_K$ are twice continuously differentiable for $p\geq 2$. 
But for $p<2$ the restriction $f|_K$ is twice continuously differentiable
only on the set open subset $U:=\cap_{j\not\in I}\{ x \in K: \xi_j^Tx+\eta_j \ne 0\}  \subset K$. 

If $x_{t_0}\ne x_{t_1}$ and $\xi_{j_0}^Tx_{t_0}=\xi_{j_1}^Tx_{t_1} =0$ for  $j_0,j_1\not\in I$,
then $j_0\ne j_1$. 
Since there are only finitely many indices, the intersection $U\cap\{ x_t:1\leq t\leq 1\}$ 
is an open nonempty subset of $\{ x_t:1\leq t\leq 1\}$.

Given $x\in K$ and $v\in\R^n$, the sum $x+v$ is an element of $K$ if and only if
$\xi_j^Tv=0$ for all $j\in I$.
The set $V:=\cap_{j\in I} \{ w\in\R^n:\xi_j^Tw=0\}$ is a vector space.
For any $x\in U$ and $0\ne v\in V$ we consider the function $h(t) :=f(x+tv)$. It holds
$$h''(0) = p(p-1)\sum_{j\not\in I} |\xi_j^Tx+\eta_j|^{p-2} (\xi_j^Tv)^2>0.$$
This means that $U\cap\{ x_t:1\leq t\leq 1\} \ne\emptyset$ contains exactly one point.
Thus, $x_0=x_1$.
\end{proof}
\subsection{Minimal algebraic $L^p$-distance}
Given points $p_j=(x_j,y_j)^T\in \R^2$, $j=1,\ldots,m$, we want to determine the linear function 
$y(x)=ax+b$, $a,b\in\R$, with minimal algebraic $L^p$-distance to the set $\{Êp_1,\ldots,p_m\}$, 
i.e., the minimum of the function
$$f:\R^2\to\R, \quad f(a,b)=\sum_{j=1}^m  | y_j-(ax_j+b)| ^p.$$
As for the investigation of the algebraic $L^1$-, $L^2$- and $L^\infty$-distance we additionally assume
that  $x_1,\ldots,x_m$ are pairwise distinct.

The function $f$ is continuous, continuously differentiable and convex.
Defining the decomposition $J_+:=\{ j:y_j>y(x_j)\}$, $J_0:=\{  j:y_j=y(x_j)\}$ and $J_-:=\{  j:y_j<y(x_j)\}$
for every linear function $y(x) = ax+b$ we obtain the partial derivations
\begin{align*}
f_b(a,b) & = \sum_{j\in J_-} p|y_j-ax_j-b|^{p-1}-\sum_{j\in J_+} p|y_j-ax_j-b|^{p-1} \\
f_a(a,b) & = \sum_{j\in J_-} px_j|y_j-ax_j-b|^{p-1}-\sum_{j\in J_+} px_j|y_j-ax_j-b|^{p-1}.
\end{align*}
Thus, $f(a,b)$ is a global minimum  if and only if
\begin{equation}
\label{eq:alg-Lp-fa}
\sum_{j\in J_-} |y_j-ax_j-b|^{p-1} = \sum_{j\in J_+} |y_j-ax_j-b|^{p-1} 
\end{equation}
and
\begin{equation}
\label{eq:alg-Lp-fb}
\sum_{j\in J_-} x_j|y_j-ax_j-b|^{p-1} = \sum_{j\in J_+} x_j|y_j-ax_j-b|^{p-1} .
\end{equation}
If $p=2n$ with $n\in\N$ and $n\geq 1$, then $f(a,b) = \sum_{j=1}^m (y_j-ax_j-b)^{2n}$,
$f_a(a,b) = p\sum_{j=1}^mx_j(y_j-ax_j-b)^{2n-1}$ and  $f_b(a,b) = p\sum_{j=1}^m(y_j-ax_j-b)^{2n-1}$.
In this case the equations (\ref{eq:alg-Lp-fa}) and (\ref{eq:alg-Lp-fb}) are
polynomials in the variables $a$ and $b$ of degree $p-1$.
Even if it possible to eliminate one of the variables, there is no general closed formula 
that expresses the common solutions of the equations (\ref{eq:alg-Lp-fa}) and (\ref{eq:alg-Lp-fb})
for $n>1$ in terms of the coordinates $x_j$, $y_j$.
\begin{theorem}
\label{satz:alg-Lp}
Let $p_j=(x_j,y_j)^T$, $j=1,\ldots,m$ such that $x_j\ne x_k$ for all $j\ne k$.

For any $p>1$ there exists a unique linear function with minimal algebraic $L^p$-distance
to the set $\{p_1,\ldots,p_m\}$.
A linear funktion $y(x) = ax+b$ with $a,b\in\R$ has minimal algebraic $L^p$-distance to the set
$\{p_1,\ldots,p_m\}$ if and only if
$a$ and $b$ satisfy the equations (\ref{eq:alg-Lp-fa}) and (\ref{eq:alg-Lp-fb}).
\end{theorem}
\begin{proof}
The assertion follows from Lemma \ref{lem:Lp-Existenz-globMin} and 
Lemma \ref{lem:Lp-Eindeutigkeit-globMin} with $\eta_j=y_j$ and $\xi_j = (-x_j,-1)^T$.
The condition $\cap_{j=1}^m \{ (a,b)^T \in \R^2: \xi_j^T(a,b)^T=0\} = \{ 0\}$ is satisfied, since
$j>1$ and $x_j\ne x_k$ for all $k\ne j$.
\end{proof}
\begin{cor}
\label{folg:alg-Lp-Spiegelsymmetrie}
If the set $\{ p_1,\ldots,p_m\}$ is invariant under the reflection in the line given by $x=\bar{x}$,
then the unique line with minimal algebraic $L^\infty$-distance to $\{ p_1,\ldots,p_m\}$ is the graph
of the linear function
$$y(x) \equiv b_0 \text{ where} \sum_{y_j>b_0}(y_j-b_0)^{p-1} = \sum_{y_j<b_0} (b_0-y_j)^{p-1}.$$
\end{cor}
\begin{proof}
Since there is a unique optimal linear function, 
the proof is similar to that of Corollary \ref{folg:algL2Spiegelsymmetrie}. 
The condition on $b_0$ follows from equation (\ref{eq:alg-Lp-fa}).
\end{proof}
\begin{rem}
Theorem \ref{satz:alg-Lp} remains true if the condition $x_j\ne x_k$ for all $j\ne k$ 
is weakened to the existence of indices $j\ne k$ such that $x_j\ne x_k$,
since even in that cases $\cap_{j=1}^m \{ (a,b)^T \in \R^2: \xi_j^T(a,b)^T=0\} = \{ 0\}$.

However, if $x_1=\ldots=x_m=\bar{x}$, then the values of $f(a,b-a\bar{x})$ 
are independent from $a$ and there exists a unique $b_0\in\R$ such that 
$f(0,b_0) \leq f(0,b)$ for all $b\in\R$. Thus, there are infinitely many linear functions with
minimal algebraic $L^p$-distance to $\{p_1,\ldots,p_m\}$. 
These are exactly $y(x) = a(x-\bar{x})+b_0$ with $a\in\R$.
\end{rem}
\subsection{Minimal geometric $L^p$-distance}
Given pairwise distinct points $p_j\in \R^2$, $j=1,\ldots,m$, we want to determine the lines
given by $g=\{ q\in\R^2 : \langle q,n\rangle = c\}$, $c\in\R$, $n\in S^1$, with 
minimal geometric $L^p$-distance to the set $\{Êp_1,\ldots,p_m\}$, 
i.e., the minimum of the function
$$f:\R\times S^1\to\R, \quad f(c,n)=\sum_{j=1}^m  | c-\langle p_j,n\rangle|^p.$$
As for the investigation of the geometric $L^1$-, $L^2$- and $L^\infty$-distance 
we decompose the index set, this means
$J_+:=\{Êj : \langle p_j,n\rangle > c\}$, $J_0:=\{Êj : \langle p_j,n\rangle = c\}$,
$J_-:=\{Êj : \langle p_j,n\rangle < c\}$ (see Figure \ref{fig:ZerlegGeo}).
Using the parametrization $n(t) = (\cos t, \sin t)^T$ we obtain the following partial derivations of the
function $f(c,n(t))$:
\begin{align*}
f_c(c,n) & = p\sum_{j\in J_-} |c-\langle p_j,n\rangle|^{p-1} - p\sum_{j\in J_+} |c-\langle p_j,n\rangle|^{p-1}\\
f_t(c,n) & = p\sum_{j\in J_+} |c-\langle p_j,n\rangle|^{p-1}p_j^TJn 
	- p\sum_{j\in J_-} |c-\langle p_j,n\rangle|^{p-1}p_j^TJn
\end{align*}
If $f(g)$ is a local minimum and the line $g$ is given by $\langle q,n\rangle = c$
 with  $c\in\R$ and $n\in S^1$, then 
\begin{equation}
\label{eq:geo-Lp-fc}
\sum_{j\in J_-} |c-\langle p_j,n\rangle|^{p-1}  = \sum_{j\in J_+} |c-\langle p_j,n\rangle|^{p-1}
\end{equation}
and
\begin{equation}
\label{eq:geo-Lp-fn}
\sum_{j\in J_-} |c-\langle p_j,n\rangle|^{p-1}p_j^TJn 
	= \sum_{j\in J_+} |c-\langle p_j,n\rangle|^{p-1}p_j^TJn.
\end{equation}
The equations (\ref{eq:geo-Lp-fc}) and (\ref{eq:geo-Lp-fn}) are polynomials in $c$ and $n$
respectively real analytic expressions in $c$ and $t$ if $p=2n$ with $n\in\N$. 
Except for $p=2$, there is no general explicit formula of the solutions of the equation $df = (0,0)$.
\begin{theorem}
\label{satz:geo-Lp}
Let $p_j\in\R^2$, $j=1,\ldots,m$ be pairwise distinct points.

For any $p>1$ there exists a line with minimal geometric $L^p$-distance to the set $\{p_1,\ldots,p_m\}$.
If the line $g=\{ q\in\R^2 : \langle q,n\rangle = c\}$ with $c\in\R$ and $n\in S^1$ has minimal
geometric $L^p$-distance to the set $\{p_1,\ldots,p_m\}$, then $c$ and $n$ satisfy the equations
(\ref{eq:geo-Lp-fc}) and (\ref{eq:geo-Lp-fn}).

If $f(c_0,n') = f(c_1,n') = \min\{ f(c,n):c\in\R, n\in S^1\}$, then $c_0=c_1$.
\end{theorem}
\begin{proof}
Applying Lemma \ref{lem:Lp-Existenz-globMin} and Lemma \ref{lem:Lp-Eindeutigkeit-globMin} with
$\eta_j=-\langle p_j,n\rangle$ and $\xi_j = (1)$ it follows that for any $n\in S^1$ there exists a unique
$c_n\in\R$ such that $f(c_n,n) = \min\{ f(c,n) : c\in\R\}$. 
Since the function $c\mapsto f(c,n)$ is continuously differentiable, equation (\ref{eq:geo-Lp-fc}) yields 
$\sum_{j\in J_-} |c_n-\langle p_j,n\rangle|^{p-1}  = \sum_{j\in J_+} |c_n-\langle p_j,n\rangle|^{p-1}$.
In particular, $J_+,J_- \ne \{ 1,\ldots,m\}$ for any line $g = \{ q : c_n=\langle q,n\rangle \}$. 
Thus,
$$a:= \min_{j=1,\ldots,m,\, n\in S^1} \langle p_j,n\rangle \leq c_n \leq 
	\max_{j=1,\ldots,m,\, n\in S^1} \langle p_j,n\rangle =:b$$
for all $n\in S^1$. 
The continuous function $f:\R\times S^1\to\R$, $(c,n)\mapsto f(c,n)$, admits a global minimum on the 
compact set $[a,b]\times S^1$. 
\end{proof}
\begin{rem}
Theorem \ref{satz:geo-Lp} remains true if the condition $x_j\ne x_k$ for all $j\ne k$ is omitted.
\end{rem}
\begin{rem}
If the set 
$$M:=\{ (c_0,n_0): f(c_0,n_0) \leq f(c,n) \,\forall c\in\R, n\in S^1\} \subset \R\times S^1$$
is infinite, then $M$ has an accumulation point, because $M\subset [a,b]\times S^1$.
If $p=2n$ with $n\in\N$, then the existence of such a limit point in $M$ implies that for any
$n\in S^1$ there exists $c_n\in \R$ such that $(c_n,n)\in M$, since $f$ is a polynomial in $c$ and $n$.
If $p\not\in 2\N$, then, in general, the function $f(c,n)$ is not real analytic.
We denote the set of normal vectors of optimal lines by $N$, 
i.e., $N:=\{ n\in S^1: \R\times\{n\} \cap M \ne \emptyset\}$.
If  $p\in 2\N$, then $N$ is finite or $N= S^1$. Is this fact true for any $p>1$?

The lines with minimal geometric $L^p$-distance to the vertices of an equilateral triangle are identified
in \cite{AP} by taking advantage of the symmetries of the point set for all $1\leq p\leq\infty$. 
In this simplest nontrivial situation 
the set $M$ consists of three optimal lines for $p>1$ with $p\ne 2, 4/3$
and $N=S^1$ for $p=4/3$ and $p=2$.
\end{rem}
\begin{appendix}
\section{Convexity}
A subset $K\subset\R^n$ is called convex if the line segment $\{ tx_1+(1-t)x_0 : 0\leq t \leq 1\}$
is contained in $K$ for any  $x_0,x_1\in K$.
The convex hull  $H(K)$ of a set  $K\subset \R^n$ is given by
$H(K)Ê:= \{ x=tx_1+(1-t)x_0 : x_0,x_1\in K, 0\leq t \leq 1\}$.
A subset $K\subset \R^n$ is convex if and only if $K=H(K)$.
Convex sets are connected.

Let $K\subset \R^n$ be a convex set.
A function $f:K\to\R$ is called convex if the inequality
\begin{equation}
\label{eq:Definition-KonvexeFunktion}
f(tx_1+(1-t)x_0) \leq tf(x_1)+(1-t)f(x_0)
\end{equation}
holds for all $x_0,x_1\in K$ and all $0\leq t\leq 1$.
The function $f$ is called strictly convex if the inequality (\ref{eq:Definition-KonvexeFunktion}) is strict
for all $t\in (0,1)$ and $x_0\ne x_1$.
If  $K_1\subset K_2\subset \R^n$ are two convex sets  and the function  $f:K_2\to\R$ is convex,
then the restriction $f\vert_{K_1}:K_1\to\R$ is convex.
\begin{lemma}
Let $f:\R^n\to\R$ be a convex function.
The set $$M:=\{Êx\in\R^n : f(x)\leq f(y) \text{ for all } y\in\R^n\}$$ is convex.
If $f$ is strictly convex, then $M$ contains at most one element.
\end{lemma}
\begin{proof}
The following inequality holds for all $0\leq t \leq 1$ and $x_0,x_1\in\R^n$:
$$f(tx_1+(1-t)x_0) \leq t f(x_1)+(1-t)f(x_0) \leq \max\{f(x_0),f(x_1)\}$$ 
If $x_0,x_1\in M$, then $f(tx_1+(1-t)x_0) = f(x_0) = f(x_1)$, i.e., $tx_1+(1-t)x_0 \in M$.

If $f$ is strictly convex, $x_0\ne x_1\in \R^n$ and $0<t<1$, then
$$f(tx_1+(1-t)x_0) < t f(x_1)+(1-t)f(x_0) \leq \max\{ f(x_0),f(x_1)\}.$$
Consequently, $\{ x_0,x_1\} \not \subset M$.
\end{proof}
\begin{lemma}
For any $\xi\in\R^n$ the linear function $y:\R^n\to \R$ defined by $y(x) = \xi^Tx$ is convex.
\end{lemma}
\begin{proof} 
It holds
$y(tx_1+(1-t)x_0) = \xi^T(tx_1+(1-t)x_0)  = t\xi^Tx_1+(1-t)\xi^Tx_0 =  t y(x_1)+(1-t)y(x_0).$
\end{proof}
\begin{lemma}
If $f,g:\R^n\supset K \to\R$ are convex functions on a convex set $K$, 
then $f+g$ and $\max\{f,g\}$ are convex functions on $K$.

If $f$ or $g$ are additionally strictly convex, then $f+g$ is strictly convex.
If $f$ and $g$ are strictly convex, then $\max\{f,g\}$ is strictly convex.
\end{lemma}
\begin{proof}
Applying the inequality (\ref{eq:Definition-KonvexeFunktion}) to $f$ and $g$ we obtain
\begin{align*}
(f+g)(tx_1+(1-t)x_0) & = f(tx_1+(1-t)x_0)+g(tx_1+(1-t)x_0) \\
& \leq  t f(x_1)+(1-t)f(x_0) + t g(x_1)+(1-t)g(x_0) \\
& = t(f+g)(x_1)+(1-t)(f+g)(x_0) .
\end{align*}
Since $t\geq 0$ and $1-t\geq 0$, it follows that
\begin{align*}
\max\{f,g\}(tx_1+(1-t)x_0) & = \max\{f(tx_1+(1-t)x_0),g(tx_1+(1-t)x_0)\} \\
\leq & \max\{t f(x_1)+(1-t)f(x_0), t g(x_1)+(1-t)g(x_0)\} \\
\leq & t\max\{f(x_1),g(x_1)\}+(1-t)\max\{f(x_0),g(x_0)\} \\
= & t \max\{f,g\}(x_1) + (1-t) \max\{f,g\}(x_0) .
\end{align*}
\end{proof}
\begin{cor}
The function $\R\to\R$ defined by $x\mapsto |x|$ is convex.
\end{cor}
\begin{proof}
$|x| = \max\{ x,-x\}$
\end{proof}
\begin{lemma}
Let $K_f\subset \R$ and $K_g\subset\R^n$ be convex sets and
$f:K_f\to\R$ and $g:K_g\to\R$ be convex functions satisfying $g(K_g)\subset K_f$.
If $f$ is an increasing function on $K_f$, then $f\circ g:K_g\to\R$ is a convex function.
\end{lemma}
\begin{proof}
The inequality $g(tx_1+(1-t)x_0) \leq tg(x_1)+(1-t)g(x_0)$ holds 
for all $x_0,x_1\in K_g$ and for all $0\leq t\leq 1$.
The monotony and the convexity of $f$ imply
\begin{align*}
(f\circ g)(tx_1+(1-t)x_0) &  = f(g(tx_1+(1-t)x_0)) \leq f(tg(x_1)+(1-t)g(x_0)) \\
& \leq tf(g(x_1))+(1-t)f(g(x_0)) \\
& = t (f\circ g)(x_1)+(1-t) (f\circ g)(x_0).
\end{align*}
\end{proof}
\begin{lemma}
Let $K\subset\R$ be a convex set.
If $f:K\to\R$ is a convex twice continuously differentiable function, then $f''(x)\geq 0$ for all $x\in\R$.
\end{lemma}
\begin{proof}
For any $a\in K$ we consider the function $h(x) := f(x) - (x-a)f'(a)-f(a)$.
Now $h(a)=0$, $h'(a) = 0$, $h''(a) =f''(a)$.
The function $f$ is convex if and only if  $h$ is convex.
The convexity of $h$ implies  
$$h(a) = h\left(\frac{a-\varepsilon}{2}+\frac{a+\varepsilon}{2}\right) \leq 
	\frac{1}{2}(h(a-\varepsilon)+h(a+\varepsilon)) \leq \max((h(a-\varepsilon),h(a+\varepsilon)) $$
for all $\varepsilon$.
Hence, $h(a)$ is not a strict local maximum and $h''(a)\geq 0$.
\end{proof}
\begin{lemma}
Let $K\subset \R$ be a convex set and $f:K\to\R$ be a convex function.
If $f''(x)\geq 0$ for all $x\in K$, then $f$ is convex.
If $f''(x) > 0$ for all $x\in K$, then $f$ is strictly convex.
\end{lemma}
\begin{proof}
For any $x_0,x_1\in K$ and $0<t<1$ we set $x_t:=tx_1+(1-t)x_0$ and consider the function
$h(x) := f(x) -(x-x_t)f'(x_t)- f(x_t)$. 
Now $h(x_t)=0$, $h'(x_t)=0$ and $h''(x)=f''(x)$. 

If $h''(x)\geq 0$ for all $x\in K$, then $h'(x)$ is increasing. There exist $a,b \in K$ such that
$a\leq x_t\leq b$, $h'(x)=0$ for all $a\leq x\leq b$, $h'(x)<0$ for all $x<a$ and 
$h'(x)>0$ for all $x>b$. Hence, $h(x)\geq 0$ for all $x\in K$.
Now $th(x_1)+(1-t)h(x_0) \geq 0$ implies
\begin{align*}
0 & \leq t(f(x_1) - (x_1-x_t) f'(x_t) - f(x_t) ) \\
& \quad +(1-t)(f(x_0) -(x_0-x_t)f'(x_t)- f(x_t)) = t f(x_1)+(1-t)f(x_0)-f(x_t)
\end{align*}
If $h''(x)>0$ for all $x\in K$, then the inequality above is strict for $0<t<1$.
\end{proof}
\begin{cor}
For any $p>1$ the function $f:[0,\infty)\to\R$ given by $x\mapsto x^p$ is strictly convex.
\end{cor}
\begin{proof}
It holds $f''(x) = p(p-1)x^{p-2} >0$ for all $x>0$. 
If $x_0=0$, $x_1>0$ and $0<t<1$, then
$$f(tx_0+(1-t)x_1) = f((1-t)x_1) = (1-t)^p x_1^p < (1-t) f(x_1) = tf(x_0)+(1-t)f(x_1),$$
since $0<1-t<1$ and $(1-t)^p<(1-t)$.
\end{proof}
\begin{cor}
Let $K\subset\R^n$ be a convex set.
If $f:K\to\R$ is convex and twice continuously differentiable, then $H_f(x)\geq 0$ for all $x\in K$.
\end{cor}
\begin{proof}
It holds $v^TH_f(x)v = h''(0)\geq 0$ for all $x\in K$ and for all $v\in\R^n$,
since $h(t) := f(x+tv)$ is a convex function.
\end{proof}
\begin{cor}
Let $K\subset \R^n$ be a convex set and $f:K\to\R$ be a differentiable function.
If  $H_f(x)\geq 0$ for all $x\in K$, then $f$ is convex.
If $H_f(x) > 0$ for all $x\in K$, then $f$ is strictly convex.
\end{cor}
\begin{cor}
For any $\xi\in\R^n$ and any $p\geq 1$ the function 
$\R^n\to\R$ defined by $x\mapsto |\xi^Tx|^p$ is convex.
\end{cor}
\end{appendix}

\end{document}